\documentclass[11pt]{article}

\usepackage{amsfonts,amssymb,amsmath,latexsym,xcolor,epsfig}
\usepackage[textsize=scriptsize,colorinlistoftodos]{todonotes}
\usepackage{mathrsfs}
\newcommand{\irreg}{\operatorname{irreg}}
\newcommand{\diag}{\operatorname{diag}}
\newcommand{\tr}{\operatorname{tr}}
\usepackage{hyperref}
\usepackage{amsthm}
\usepackage{enumerate}
\newtheorem{theorem}{Theorem}[section]

\newtheorem{lemma}[theorem]{Lemma}
\newtheorem{corollary}[theorem]{Corollary}

\theoremstyle{definition}

\theoremstyle{remark}

\input{epsf}

\makeatletter
\def\Ddots{\mathinner{\mkern1mu\raise\p@
\vbox{\kern7\p@\hbox{.}}\mkern2mu
\raise4\p@\hbox{.}\mkern2mu\raise7\p@\hbox{.}\mkern1mu}}
\makeatother

\title{\vspace{-0.7cm}A tight lower bound for Szemer\'edi's regularity lemma}
\author{
Jacob Fox\thanks{
    Department of Mathematics,
    Massachusetts Institute of Technology,
    Cambridge, MA 02139-4307.
    Email: {\tt fox@math.mit.edu}.
    Research supported by a Packard Fellowship, by a Simons Fellowship, by NSF grant DMS-1069197, by an Alfred P. Sloan Fellowship, and by an MIT NEC Corporation Award.}  \and L\'aszl\'o Mikl\'os Lov\'asz\thanks{Department of Mathematics,
    Massachusetts Institute of Technology,
    Cambridge, MA 02139-4307.
    Email: {\tt lmlovasz@math.mit.edu}.}
}
\date{}
\begin{document}
\maketitle

\begin{abstract} 
Addressing a question of Gowers, we determine the order of the tower height for the partition size in a version of Szemer\'edi's regularity lemma. 
\end{abstract}

\section{Introduction}

Szemer\'edi's regularity lemma \cite{Sz76} is one of the most powerful tools in graph theory. An early version was used by Szemer\'edi  \cite{Sz75} in his proof of the celebrated Erd\H{o}s-Tur\'an conjecture (now known as Szemer\'edi's theorem) on long arithmetic progressions in dense subsets of the integers. The regularity lemma (see the surveys \cite{KS}, \cite{RS}) has since become a central tool in extremal combinatorics, with many applications in number theory, graph theory, and theoretical computer science. Roughly speaking, the lemma says that the vertex set of any graph may be partitioned into a small number of parts such that the bipartite subgraph between almost every pair of parts behaves in a random-like fashion.

We next describe more precisely a version of Szemer\'edi's regularity lemma. This version was first formulated by Lov\'asz and Szegedy \cite{LS07}, and can easily be shown to be equivalent to Szemer\'edi's original version. For a pair of vertex subsets $X$ and $Y$ of a graph $G$, let $e(X,Y)$ be the number of ordered pairs of vertices $(x,y) \in X \times Y$ that have an edge between them in the graph. Let $d(X,Y)=\frac{e(X,Y)}{|X||Y|}$ be the edge density between $X$ and $Y$.
The \emph{irregularity} of the pair $X,Y$ is defined to be 
\[\irreg(X,Y)=\max_{U \subset X,W \subset Y}\bigl |e(U,W)-|U||W|d(X,Y)\bigl |.\]
This is a value between $0$ and $|X||Y|$. If this is a small fraction of $|X||Y|$, then the edge distribution between $X$ and $Y$ is quite uniform, or random-like. The \emph{irregularity} of a partition $\mathscr{P}$ of the vertex set of $G$ is defined to be
\[\irreg(\mathscr{P})=\sum_{X,Y \in \mathscr{P}}\irreg(X,Y).
\]

Szemer\'edi's regularity lemma, as stated in \cite{LS07}, is as follows. 

\begin{theorem}\label{regver}

For any $\epsilon>0$, there is a (least) $M(\epsilon)$ such that any graph $G=(V,E)$ has a vertex partition into at most $M(\epsilon)$ parts with irregularity at most $\epsilon|V|^2$.

\end{theorem}

We present for completeness the standard proof of the regularity lemma in Section \ref{upper}, using a density increment argument with the mean square density. It shows that $M(\epsilon)$ is at most a tower of twos of height $O(\epsilon^{-2})$. By a careful argument, we obtain a tower height which is at most $2+\epsilon^{-2}/16$. Formally, the \emph{tower function $T(n)$ of height $n$} is defined recursively by $T(1)=2$ and $T(n)=2^{T(n-1)}$. 

A vertex partition of a graph is \emph{equitable} if any two parts differ in size by at most one.  In the statement of the regularity lemma, it is often added that the vertex partition is equitable.  There are several good reasons not to add this requirement to the regularity lemma. First, our main result, which gives a lower bound on $M(\epsilon)$ whose height is on the same order as the upper bound, does not need this requirement. Second, the proof of the upper bound is cleaner without it. Finally, it is further shown in \cite{FGLZ} that whether or not an equitable partition is required has a negligible effect on $M(\epsilon)$.


The original version of the regularity lemma is as follows. A pair of vertex subsets $X,Y$ is {\it $\epsilon$-regular} if for all $U \subset X$ and $W \subset Y$ with $|U| \ge \epsilon |X|$ and $|W| \ge \epsilon |Y|$, we have $|d(U,W)-d(X,Y)| \le \epsilon$. It is not difficult to show that if a pair $X,Y$ is $\epsilon$-regular, then its irregularity is at most $\epsilon|X||Y|$. Conversely, if the irregularity is at most $\epsilon^3|X||Y|$, then the pair is $\epsilon$-regular. A vertex partition with $k$ parts is {\it $\epsilon$-regular} if all but at most $\epsilon k^2$ pairs of parts are $\epsilon$-regular. The regularity lemma states that for every $\epsilon>0$, there is a $K(\epsilon)$ such that there is an equitable $\epsilon$-regular partition into at most $K(\epsilon)$ parts. Again, it is not too difficult to see that if the irregularity of an equitable partition is at most $\epsilon^4 |V|^2$, then it is $\epsilon$-regular. Conversely,  if an equitable partition is $\epsilon$-regular, it has irregularity at most $2\epsilon|V|^2$. 

For many applications, it would be helpful to have a smaller bound on the number of parts in the regularity lemma. In other words, can the bound on $M(\epsilon)$ be significantly improved? A breakthrough result of Gowers \cite{Go97} gave a negative answer to this problem, showing that $K(\epsilon)$ is at least a tower of twos of height on the order of $\epsilon^{-1/16}$, and hence $M(\epsilon)$ is at least a tower of height on the order of $\epsilon^{-1/64}$.  Bollob\'as referred to this important work as a `tour de force' in  Gowers' Fields Medal citation \cite{Bo98}. Gowers \cite{Go97} further raises the problem of determining the correct order of the tower height in Szemer\'edi's regularity lemma. 

More recently, Conlon and Fox \cite{CF12} estimated the number of irregular pairs in the regularity lemma, and  Moshkovitz and Shapira \cite{MS} gave a simpler proof of a tower-type lower bound. However, these results still left a substantial gap in the order of the tower height. 

The main result in this paper is a tight lower bound on the tower height in the regularity lemma. It shows that $M(\epsilon)$ in Theorem \ref{regver}, the regularity lemma, can be bounded from below and above by a tower of twos of height on the order of $\epsilon^{-2}$. Our lower bound construction shows that the density increment argument using the mean square density in the proof of the regularity lemma cannot be improved. 


\begin{theorem}\label{main}
For $\epsilon < 1/4$, there is a constant $c>0$ such that the bound  $M(\epsilon)$ on the number of parts in the regularity lemma is at least a tower of twos of height at least $c\epsilon^{-2}$. 
\end{theorem}

The assumption $\epsilon<1/4$ is needed because it is easy to check that $M(\epsilon)=1$ for $\epsilon \geq  1/4$ and $M(\epsilon) \geq 2$ otherwise. The proof shows that we may take $c=10^{-26}$, although we do not try to optimize this constant in order to give a clearer presentation.  

All of the proofs of tower-type lower bounds on the regularity lemma, including that of Theorem \ref{main}, build on the basic framework developed by Gowers in \cite{Go97}. It involves a probabilistic construction of a graph which requires many parts in any partition with small irregularity. A simple argument shows that it is sufficient for the construction to be an edge-weighted graph, with edge weights in $[0,1]$. The construction begins with a sequence of equitable vertex partitions $\mathscr{P}_0,\mathscr{P}_1,\ldots,\mathscr{P}_s$ with $s=\Theta(\epsilon^{-2})$ and $\mathscr{P}_{i+1}$ a refinement of $\mathscr{P}_i$ with exponentionally more parts than $\mathscr{P}_i$. The weight of an edge will depend on which parts its vertices lie in. The proof further shows that {\it any} partition of the constructed graph with irregularity at most $\epsilon$ cannot be too far from being a refinement of $\mathscr{P}_s$, which has many parts. The theorem quickly follows from this result. 

There are two novelties in the proof of Theorem \ref{main}. While the previous constructions could not give a tight bound, the lower bound construction here is carefully chosen to mimic the upper bound. However, the main novelty is in how we analyze the construction. To get a good lower bound on the irregularity of a partition, it is helpful to see how each part $S$ breaks into smaller pieces based on each partition $\mathscr{P}_i$. A careful bookkeeping allows us to lower bound the irregularity by collecting how much $S$ splits at each step $i$. Previous arguments could not collect the contributions to the irregularity from a part splitting off in pieces over different steps. 

To accomplish this, we need careful averaging arguments over many parts to obtain the desired bound. This type of averaging argument is in fact necessary as we cannot guarantee irregularity from a particular part $S$. Indeed, it is not too difficult to show that for each $\epsilon>0$ there is a $k=k(\epsilon)$ which is only exponential in $\epsilon^{-O(1)}$ such that every graph has an equitable vertex partition into parts $V_1,\ldots,V_k$ such that $V_1$ is $\epsilon$-regular with every other part $V_j$. This justifies the need for the more global averaging arguments we use to bound the irregularity. 
 
We remark that Theorem \ref{regver} is often more convenient to work with than the original version of the regularity lemma. There are several reasons for this. One reason is that in various applications, the standard notion of regularity requires considering extra case analysis depending on whether or not certain subsets are at least an $\epsilon$-fraction of a part. This is so that one can apply the density conclusion in the definition of an $\epsilon$-regular pair. A good example is the proof of the counting lemma, an important tool in combination with the regularity lemma. See \cite{BCLSV} for a nice proof of the counting lemma using the notion of irregularity. Another reason is that the notion of irregularity is closely related to the cut norm developed in the proof of the Frieze-Kannan weak regularity lemma, which begets the close relationship between these two important result; see \cite{BCLSV,CF12,LS07,RS,Tao06} for details.

\noindent {\bf Organization.} In the next section, we present a proof of Szemer\'edi's regularity lemma, Theorem \ref{regver}. In Section \ref{construction}, we present our construction for the lower bound. In Section \ref{sing}, we review some properties of weighted graphs that follow from their singular values, including a bipartite analogue of the expander mixing lemma. In Section \ref{graphchangesec}, we use these results to show where most of the irregularity of the construction comes from. In Section \ref{secmain}, we use this to show that our construction requires a large number of parts for any partition of small irregularity. In the last section, we make some concluding remarks. For the sake of clarity of presentation, we do not make any serious attempt to optimize absolute constants in our statements and proofs.

\section{Regularity Lemma Upper Bound} \label{upper}

Here, for completeness, we present the proof of Szemer\'edi's regularity lemma, showing that the bound $M(\epsilon)$ on the number of parts is at most a tower of twos of height at most $2+ \epsilon^{-2}/16$. The proof is similar to Szemer\'edi's original proof \cite{Sz76}, although it uses a probabilistic approach as in \cite{ASbook}.

The key idea that makes the proof work is to use a density increment argument with the \emph{mean square density}. Let $G=(V,E)$ be a graph and let $\mathscr{P}$ be a vertex partition into parts $V_1,V_2,\ldots,V_k$. The mean square density of the partition $\mathscr{P}$ is defined to be 
\[q(\mathscr{P}):=\sum_{i,j=1}^k \frac{|V_i||V_j|}{|V|^2} d(V_i,V_j)^2.\]

Let $x$ and $y$ be two vertices of $G$, chosen independently and uniformly at random. Let $Z$ be the random variable which is the density  $d(V_i,V_j)$ between the pair of parts for which $x \in V_i$ and $y \in V_j$. Note that the mean square density $q(\mathscr{P})$ is exactly $\mathbb{E}[Z^2]$. Now, suppose that we have a refinement $\mathscr{P}'$ of $\mathscr{P}$. Let $Z'$ be the corresponding random variable for $\mathscr{P}'$. It is not difficult to see that for any fixed $i,j$, we have $\mathbb{E}(Z' | x \in V_i,y \in V_j)=d(V_i,V_j)$, which is equal to $Z$ conditioned on $x \in V_i$ and $y \in V_j$. Thus, 
\begin{align} \label{dif} \mathbb{E}(Z'^2-Z^2|x \in V_i,y \in V_j) & = \mathbb{E}((Z'-Z)^2+2Z(Z'-Z)|x \in V_i,y \in V_j) \nonumber \\ & = \mathbb{E}((Z'-Z)^2| x \in V_i, y \in V_j)+\mathbb{E}(2Z(Z'-Z)|x \in V_i,y \in V_j) \nonumber \\ & = \mathbb{E}((Z'-Z)^2| x \in V_i, y \in V_j).\end{align}


This is clearly non-negative for any pair of parts $V_i,V_j$. This implies that the mean square density cannot decrease from taking a refinement. However, the next lemma shows that if we take the refinement $\mathscr{P}'$ carefully, we can find a lower bound on the difference in terms of the irregularity of $\mathscr{P}$.

\begin{lemma}\label{regprooflem}
Suppose we have a partition $\mathscr{P}$ into $k$ parts, with irregularity $z|V|^2$. Then, there is a refinement $\mathscr{P}'$ of $\mathscr{P}$ with at most $k2^{k+1}$ parts and $q(\mathscr{P}') \geq q(\mathscr{P})+4z^2$.
\end{lemma}

\begin{proof} Let  $V_1,V_2,\ldots,V_k$ be the parts of $\mathscr{P}$. 
Fix a pair of parts $V_i$, $V_j$. First, assume $i \ne j$ (this just makes the argument a bit simpler). Let the irregularity of $V_i$ and $V_j$ be $z_{ij}|V_i||V_j|$, given by subsets $W_i \subset V_i$, $W_j \subset V_j$. Hence, we have 
\[z_{ij}=\frac{1}{|V_i||V_j|}\big|e(W_i,W_j)-|W_i||W_j|d(V_i,V_j) \big |=\frac{|W_i||W_j|}{|V_i||V_j|}\big |d(W_i,W_j)-d(V_i,V_j)\big |.\]
For simplicity of notation, we will assume that the value on the right hand side in the absolute value is positive; the exact same proof works if it is negative. For now, let $\mathscr{P}'$ be the refinement of $\mathscr{P}$ obtained by dividing $V_i$ into $W_i$ and $U_i=V_i \setminus W_i$, and $V_j$ into $W_j$ and $U_j=V_j \setminus W_j$, and keeping the rest of the parts the same. Let $Z$ and $Z'$ be defined as before, and let $\widetilde{Z}$ be the random variable equal to $Z'-Z$ conditioned on $x \in V_i,y \in V_j$. Using \eqref{dif}, we have  
\[\mathbb{E}(Z'^2-Z^2|x \in V_i,y \in V_j)=\mathbb{E}((Z'-Z)^2|x \in V_i,y \in V_j)=\mathbb{E}(\widetilde{Z}^2).\]

If $x \in V_i,y \in V_j$, then $Z=d(V_i,V_j)$. However, if $x \in W_i,y \in W_j$, then $Z'=d(W_i,W_j)$ and  $\widetilde{Z}=d(W_i,W_j)-d(V_i,V_j)=z_{ij}\frac{|V_i||V_j|}{|W_i||W_j|}$.

We use the following simple fact. If $T$ is a random variable with $\mathbb{E}(T)=0$ and $T= a \neq 0$ with probability $p<1$, then $\mathbb{E}(T^2) \ge \frac{p}{1-p}a^2$. Indeed, if we let $x=\mathbb{E}(T|T \ne a)$, then we have \[ 0=\mathbb{E}(T)=pa+(1-p)x,\] which gives \[x=\frac{-pa}{1-p},\] and so \[ \mathbb{E}(T^2)=pa^2+(1-p)\mathbb{E}(T^2|T \ne a) \ge pa^2+(1-p)x^2= pa^2+\frac{p^2a^2}{1-p} =\frac{p}{1-p}a^2,\] where the inequality is by an application of the Cauchy-Schwarz inequality.

We know that if $x \in W_i,y \in W_j$, then $\widetilde{Z}$ is equal to $z_{ij}\frac{|V_i||V_j|}{|W_i||W_j|}$. Let $w=|W_i||W_j|$ and $v=|V_i||V_j|$ Applying the above statement for $T=\widetilde{Z}$, $a=z_{ij}\frac{v}{w}$, and $p=\frac{w}{v}$, we obtain that 
\[\mathbb{E}(Z'^2-Z^2|x \in V_i,y \in V_j)=\mathbb{E}(\widetilde{Z}^2) \ge z_{ij}^2 (v^2/w^2)\frac{w/v}{1-w/v} = z_{ij}^2 \frac{1}{\frac{w}{v}(1-\frac{w}{v})}  \ge 4 z_{ij}^2. \]

Note that the above is true if we just assume that the partition is a further refinement of $\mathscr{P}'$. Indeed, if $\mathscr{P}''$ is the refinement and $Z''$ the corresponding random variable, then 
\[\mathbb{E}(Z''^2-Z^2|x \in V_i,y \in V_j)=\mathbb{E}(Z''^2-Z'^2|x \in V_i,y \in V_j)+\mathbb{E}(Z'^2-Z^2|x \in V_i,y \in V_j) \ge 0+4z_{ij}^2. \]

We assumed above that $i \ne j$. If $i=j$ the same argument works by dividing $V_i$ into four parts, and keeping the other parts the same.

Now, let us do this division for every pair of parts $V_i,V_j$, and take the common refinement. With a slight abuse of notation, call this common refinement $\mathscr{P}'$, and take $Z'$ to be the corresponding random variable. Hence, the increase in the mean square density is 
\[q(\mathscr{P}')-q(\mathscr{P})=\mathbb{E}(Z'^2)-\mathbb{E}(Z^2)=\sum_{i,j=1}^k \mathbb{E}(Z'^2-Z^2|x \in V_i,y \in V_j)\frac{|V_i||V_j|}{|V|^2} \ge \sum_{i,j=1}^k 4 z_{ij}^2\frac{|V_i||V_j|}{|V|^2}.\]

By the Cauchy-Schwarz inequality, we have
\[\left(z|V|^2\right)^2=\left(\sum_{i,j=1}^k z_{ij}|V_i||V_j|\right)^2 \le \left(\sum_{i,j=1}^kz_{ij}^2|V_i||V_j|\right)\left(\sum_{i,j=1}^k |V_i||V_j|\right)=\left(\sum_{i,j=1}^kz_{ij}^2|V_i||V_j|\right)|V|^2.\]

Dividing this by $|V|^4/4$, we obtain that the increase in the mean square density is at least
\[4\sum_{i,j=1}^k z_{ij}^2\frac{|V_i||V_j|}{|V|^2} \ge 4z^2.\]

The part $V_i$ is partitioned into two parts for each $j \not = i$, and into four parts if $j=i$. As there are $k$ parts in $\mathscr{P}$, each part $V_i$ is divided into at most $2^{k+1}$ parts in $\mathscr{P}'$, giving a total of at most $k2^{k+1}$ parts in $\mathscr{P}'$, and completing the proof of the lemma.
\end{proof}

\begin{proof}[Proof of Theorem \ref{regver}]

Let $d=d(V,V)$ be the edge density of the graph $G=(V,E)$. First, note that if we take the trivial partition with one part, then the mean square density is $d^2$. If we take a partition into parts of size one, then the mean square density is $d$. As taking a refinement of a partition can not decrease the mean square density, the mean square density of any vertex partition of $G$ is always between $d^2$ and $d$. Thus, the mean square density of every vertex partition lies in the interval $[d^2,d]$ of length $d-d^2=d(1-d) \le 1/4$. Now, the proof is as follows.

Let $\mathscr{P}_0$ be the trivial partition with one part. We will recursively define a sequence of refinements $\mathscr{P}_0,\mathscr{P}_1,\ldots$ as follows. If the partition $\mathscr{P}_i$ has irregularity at most $\epsilon |V|^2$, then this is the desired partition, and we are done. Otherwise, letting $k$ denote the number of parts of $\mathscr{P}_i$, by Lemma \ref{regprooflem} there is a refinement $\mathscr{P}_{i+1}$ into at most $k2^{k+1}$ parts such that $q(\mathscr{P}_{i+1}) \geq q(\mathscr{P}_i)+4\epsilon^{2}$. Since the mean square density cannot be more than $q(\mathscr{P}_0)+1/4$, this iteration can happen for at most $\lfloor \epsilon^{-2}/16 \rfloor$ steps, and we obtain a partition with irregularity at most $\epsilon|V|^2$. Thus, the number of parts is at most $k_s$ with $s=\lfloor \epsilon^{-2}/16 \rfloor$, where $k_i$ is defined recursively by $k_0=1$ and $k_{i+1}=k_i2^{k_i+1}$. At each step, we gain one exponential in the number of parts, and one can easily show by induction that $k_i \le T(i+2)/4$, where $T$ is the tower function defined in the introduction. Indeed, $k_0=1=T(2)/4$, and by induction \[k_{i+1}=k_i2^{k_i+1} \le 2^{2k_i} \le 2^{4k_i-2} \le 2^{T(i+2)-2}=\frac{T(i+3)}{4}.\] Thus, the total tower height is at most $2+\epsilon^{-2}/16$.
\end{proof}

\section{Construction} \label{construction}

We next give the construction of the graph which we use to prove Theorem \ref{main}. We will actually construct a weighted graph $G$ with edge weights in $[0,1]$. The following argument is based on a similar one in \cite{Go97}, but adapted to handle irregularity. It shows that constructing a weighted graph with the desired properties (rather than an unweighted graph) is sufficient. We first prove the following lemma. 

\begin{lemma} \label{countclose}
Suppose we have a weighted graph $G$ on $N$ vertices, with weights from the interval $[0,1]$. Let $\widetilde{G}$ be a random, unweighted graph on the same set of vertices, where each pair of vertices forms an edge with probability equal to its weight, independently of the other pairs. Then, with positive probability, $|e_{\widetilde{G}}(A,B)-e_G(A,B)| \le 4N^{3/2}$ for every pair $A,B$ of vertex subsets.
\end{lemma}

\begin{proof}
We may assume $N \geq 3$ as otherwise the lemma is trivial. For a pair of vertices $x,y$, let $b(x,y)=\widetilde{G}(x,y)-G(x,y)$. Then $\mathbb{E}(b(x,y))=0$, and $|b(x,y)| \le 1$, for each $x,y$. Fix two subsets $A$ and $B$. Then \[e_{\widetilde{G}}(A,B)-e_G(A,B)=\sum_{\substack{x \in A\\y \in B}} b(x,y).\] 
Let $t=4N^{3/2}$. We can apply Azuma's inequality to show that 
\[P(|e_{\widetilde{G}}(A,B)-e_G(A,B)|>t)=P(|\sum_{\substack{x \in A\\y \in B}} b(x,y)|>t) \le 2e^{-t^2/(8|A||B|)} \le 2e^{-t^2/(8N^2)} =2e^{-2N}. \]
Note that there is an $8$ in the exponent instead of a $2$ because $A$ and $B$ may intersect.

As there are $2^{2N}$ pairs of sets of vertices $A$ and $B$, the probability that there is a pair of sets $A,B$ such that $|e_{\widetilde{G}}(A,B)-e_G(A,B)|> 4 N^{3/2}$ is at most 
\[2^{2N} 2e^{-2N} = e^{(2N+1)\ln 2 - 2N}=e^{\ln 2-2(1- \ln 2)N}<1.\]
Hence, with positive probability, we have $|e_{\widetilde{G}}(A,B)-e_{G}(A,B)| \le 4 N^{3/2}$ for every pair of sets $A,B$.  
\end{proof}

Next, we show that if we replace a weighted graph with an unweighted graph whose existence is guaranteed by the previous lemma, then the irregularity of any partition cannot decrease by too much. We note that an unweighted graph can be thought of as a weighted graph with weights $0$ and $1$.

\begin{lemma} \label{irregclose}
Suppose we have two weighted graphs $G$ and $G'$, on the same set $V$ of vertices, such that for any pair of subsets $A,B \subset V$, we have $|e_G(A,B)-e_{G'}(A,B)| \le t$. Then, the following holds:
\begin{enumerate}
\item For any pair of subsets $U,W \subset V$, we have \[|\irreg_G(U,W)-\irreg_{G'}(U,W)| \le 2t.\]

\item For any positive integer $k$ and partition $\mathscr{P}$ with  at most $k$ parts, \[|\irreg_G(\mathscr{P})-\irreg_{G'}(\mathscr{P})| \le 2k^2t.\]

\end{enumerate}

\end{lemma}

\begin{proof}

For part 1, by symmetry between $G$ and $G'$, it suffices to show that \[\irreg_G(U,W)-\irreg_{G'}(U,W) \le 2t,\]
or equivalently, $\irreg_{G'}(U,W) \ge \irreg_G(U,W)-2t$.

By the definition of irregularity, there are subsets $U_1 \subset U$ and $W_1 \subset W$ which satisfy \[\left |e_G(U_1,W_1)-\frac{|U_1||W_1|}{|U||W|}e_{G}(U,W)\right |=\irreg_G(U,W).\]
As each of the terms on the left hand side changes by at most $t$ in changing $G$ to $G'$, the difference changes by at most $2t$. 
Therefore, \[ \left |e_{G'}(U_1,W_1)-\frac{|U_1||W_1|}{|U||W|}e_{G'}(U,W) \right | \ge \irreg_{G}(U,W)-2t.\]
Since $\irreg_{G'}(U,W)$ is the maximum of the left side over all sets $U_1 \subset U,W_1 \subset W$, this implies that \[\irreg_{G'}(U,W) \ge \irreg_G(U,W)-2t,\] completing the proof of part 1.

For part 2, let $\mathscr{P}$ partition $V$ into parts $V_1,V_2,\ldots,V_l$ with $l \le k$. Then \[\irreg_G(\mathscr{P})=\sum_{i,j=1}^l \irreg_G(V_i,V_j) .\] 
The definition for $G'$ is analogous. Thus, \begin{align*} 
 \left | \irreg_G(\mathscr{P})-\irreg_{G'}(\mathscr{P})\right | & = \left |\sum_{i,j=1}^l \irreg_G(V_i,V_j)-\irreg_{G'}(V_i,V_j) \right |\\
&\le \sum_{i,j=1}^l \left | \irreg_G(V_i,V_j)-\irreg_{G'}(V_i,V_j) \right| \le l^2 2t \le 2k^2t.
\end{align*}

\end{proof}

Combining the previous two lemmas, we have the following immediate corollary. 

\begin{corollary}\label{cortwolem}
For every weighted graph $G$ on $N$ vertices there is an unweighted graph $G'$ on the same set of vertices satisfying $|\irreg_{G'}(\mathscr{P})-\irreg_{G}(\mathscr{P})|\leq 8k^2N^{3/2}$ for every vertex partition $\mathscr{P}$ with at most $k$ parts.\end{corollary}

Now, we will construct a weighted graph on $N$ vertices 
such that for any partition into at most $k$ parts, where $k$ is a tower of twos of height $10^{-26} \epsilon^{-2}$, the irregularity is at least 
$\epsilon N^2$. It will be clear from the construction that we may take $N$ to be arbitrarily large. Let $0<\gamma<1/2$. By Corollary \ref{cortwolem}, if we take $N \geq 64\gamma^{-2} \epsilon^{-2}k^4$, then $8k^2N^{3/2} \leq \gamma \epsilon N^2$, and we obtain an unweighted graph such that the irregularity is at least $(1-\gamma)\epsilon N^2$ in any partition into at most $k$ parts. This justifies why it is sufficient to construct a weighted graph. If $\epsilon \geq 10^{-13}$, the height of the tower is at most $1$, and the result follows from the discussion immediately after Theorem \ref{main}. Hence, we can and will assume that $\epsilon < 10^{-13}$.

The weighted graph $G$ we construct to get a lower bound on $M(\epsilon)$ is bipartite. This simplifies the analysis of the construction, and does not affect the constants by too much. For the edge-weighted bipartite graph $G$ between vertex sets $V$ and $W$, each of equal size, we will have a sequence of equitable vertex partitions $\mathscr{P}_0,\mathscr{P}_1,\ldots,\mathscr{P}_s$ of $V$, and $\mathscr{Q}_0,\mathscr{Q}_1,\ldots,\mathscr{Q}_s$ of $W$ with $\mathscr{P}_{i+1}$ a refinement of $\mathscr{P}_i$, $\mathscr{Q}_{i+1}$ a refinement of $\mathscr{Q}_i$ for $0 \leq i \leq s-1$, $|\mathscr{P}_i|=|\mathscr{Q}_i|$ for $0 \leq i \leq s$, and the number of parts of $\mathscr{P}_{i+1}$ is exponential in the number of parts of $\mathscr{P}_i$. More precisely, we have a sequence $x_i$, and we will divide each part of $\mathscr{P}_{i-1}$ and $\mathscr{Q}_{i-1}$ into $2x_i$ equal parts. We let $k_i$ be the number of parts of $\mathscr{P}_i$, so $k_i=2x_ik_{i-1}$. We start with $k_0=1$, $x_1=2^{10}$, and let $x_{i+1}=2^{x_i/16}$ for $i \geq 1$. For example, it follows that $x_2=2^{2^6}$ and $x_3=2^{2^{60}}$. Note that we did not say anything about the number of vertices in the parts of the last partition, which can be any positive integer. Thus, the number of vertices of the graph can be arbitrarily large, and thus the argument in the previous paragraph does indeed work.

Let $\alpha$  be the minimum number with $\alpha > 2^{26}\cdot 10000 \epsilon$ and $\alpha^{-1}$ is a multiple of $6$ (this will make our calculations later simpler). Since $\epsilon<10^{-13}$, we have that $2^{26}\cdot 10000 \epsilon<1/6$, and this implies that $\alpha \le 2^{27}\cdot 10000 \epsilon$. We take $s=\alpha^{-2}/36$. Thus, $s \ge \frac{1}{36} \cdot \frac{1}{2^{54}\cdot 10^8}\epsilon^{-2}\ge 10^{-26} \epsilon^{-2}$. Note that we only need to specify the edge weights between $V$ and $W$ because 
the non-edges (those pairs inside $V$ or inside $W$) have weight $0$. 

We begin with a weighted bipartite graph $G_0$ which has constant weight $1/2$ between $V$ and $W$, and take $\mathscr{P}_0$ and $\mathscr{Q}_0$ to be the trivial partitions of each side into a single part. For each $i$ from $1$ to $s$, we will construct a weighted bipartite graph $G_i$ between $V$ and $W$ with every edge edge weight equal to $-\alpha$, $0$, or $\alpha$. The graph $G_i$ will have the property that it is constant between any part in $\mathscr{P}_i$ and any part in $\mathscr{Q}_i$. Thus, $G_i$ is a blow-up of an edge-weighted graph between $\mathscr{P}_i$ and $\mathscr{Q}_i$. We will let $\widetilde G_i=G_0+G_1+\cdots+G_i$, i.e., the edge weight of a pair in $\widetilde G_i$ is the sum, over all $j \leq i$, of that pair's edge weight in $G_j$. 
The weighted graph $G$ is defined as $G:=\widetilde G_s=G_0 +\ldots+G_{s}$, so that the edge weight in $G$ is the sum of the edge weights of the corresponding edge in each $G_i$. 

For the construction, we have left to specify the edge weights in the $G_i$ for $i \geq 1$, and we do so recursively. We do the following for each $i$ from $1$ to $s$. If the (constant) weight in $\widetilde G_{i-1}$ across a pair $X \in \mathscr{P}_{i-1}$ and $Y \in \mathscr{P}_{i-1}$ is $0$ or $1$, then we call the pair $X,Y$ inactive, and the edge weight in $G_i$ across the pair $X,Y$ is $0$. Otherwise, we call the pair $X,Y$ active, and, since $G_0$ had constant weight $1/2$, and since $\alpha^{-1}$ is an even integer, the (constant) weight between the pair $X,Y$ in $\widetilde G_{i-1}$ is a multiple of $\alpha$ which is at least $\alpha$ and at most $1-\alpha$. 

For every pair $X \in \mathscr{P}_{i-1}, Y \in \mathscr{Q}_{i-1}$ of parts, we do the following. Recall that $X$ and $Y$ are each divided into $2x_{i}$ parts in $\mathscr{P}_{i}$ and $\mathscr{Q}_{i}$, respectively.
We randomly divide the $2x_{i}$ parts in $X$ into two groups of size $x_{i}$, and let $X_Y^0$ and $X_Y^1$ be the vertices in each of these parts. Thus, $X=X_Y^0 \sqcup X_Y^1$ is an equitable partition of the vertices in $X$. We define $Y=Y_X^0 \sqcup Y_X^1$ analogously. Thus, each $Y \in \mathscr{Q}_{i-1}$ gives a random  partition of $X$, and we make these random partitions independently. If $X,Y$ is an active pair, then, for $a=0,1$, the edge weight in $G_i$ is $\alpha$ between $X_Y^a$ and $Y_X^a$, and $-\alpha$ between $X_Y^a$ and $Y_X^{1-a}$. See Figure \ref{figure0}. This completes the construction of $G$. 

\begin{figure}[h]

\centering
\begin{tikzpicture}

\draw (4,4) ellipse (1.5 and 5); 
\node at (4,9.5) {\LARGE $X$}; 
\node at (9,9.5) {\LARGE $Y$}; 
\node at (2,7) {\LARGE $X_Y^1$}; 
\node at (2,1) {\LARGE $X_Y^0$}; 
\node at (11,7) {\LARGE $Y_X^1$}; 
\node at (11,1) {\LARGE $Y_X^0$};
\draw (2.5,4) -- (5.5,4);
\draw (9,4) ellipse (1.5 and 5);
\draw (7.5,4) -- (10.5,4);
\draw[line width=4] (4,2.7) -- (9,5.7);
\node at (5.7,5.1) {\LARGE $-\alpha$};
\draw[line width=4] (4,5.7) -- (9,2.7);
\node at (5.7,3.2) {\LARGE $-\alpha$};
\draw[line width=4] (4,1.2) -- (9,1.2);
\node at (6.5,1.5) {\LARGE $+\alpha$};
\draw[line width=4] (4,7) -- (9,7);
\node at (6.5,7.3) {\LARGE $+\alpha$};

\end{tikzpicture}

\caption{An active pair of parts $X,Y$}
\label{figure0}
\end{figure}
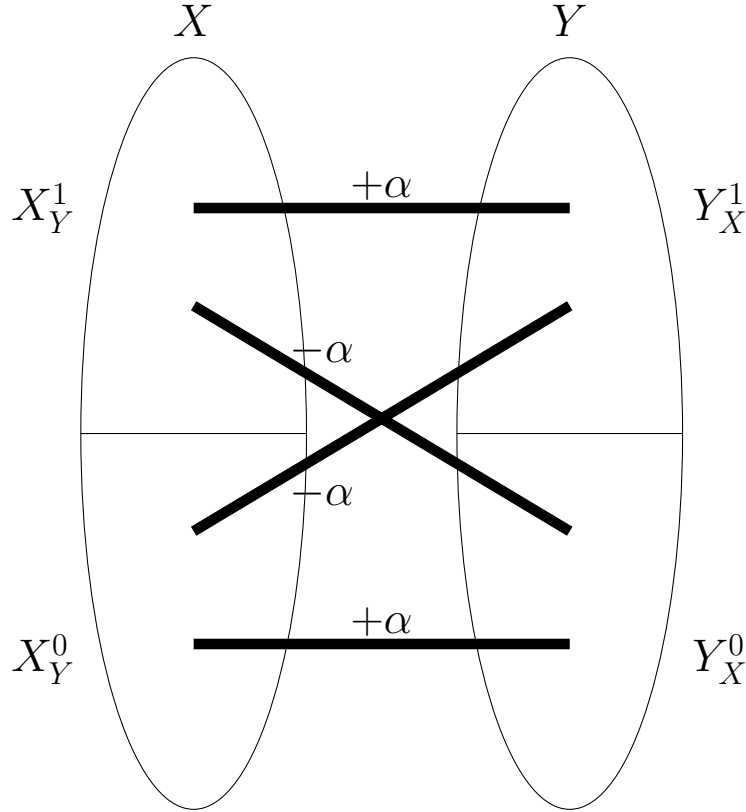

Note that if a pair of vertices goes across an inactive pair in step $i$, then its weight in $\widetilde G_j$ for $j \geq i$ (and hence in $G$) is fixed at $0$ or $1$, and it will go across inactive pairs at each later step. We chose $s=\alpha^{-2}/36$ small enough to guarantee that only a small fraction of pairs will be inactive. In fact, the following lemma is true. 

\begin{lemma} \label{inactive}
In each step $i$, for any $X \in \mathscr{P}_i$, the proportion of $Y \in \mathscr{Q}_i$ such that the pair $X,Y$ is inactive is at most $.05$, that is, there are at most $.05|\mathscr{Q}_i|$ such $Y$. The analogous statement holds for any part $Y \in \mathscr{Q}_i$.
\end{lemma}

\begin{proof}
Let $Y \in \mathscr{Q}_i$ be a part picked uniformly at random.  For each $1 \le j \le i$, let $y_j$ be the number such that $y_j\alpha$ is the  value of (any) edge between $X$ and $Y$ in the original $G_j$ (they all have the same value), if this is nonzero. This depends only on the partitions of the corresponding parts of step $j-1$, thus, we can define this even if the pair from step $j-1$ are inactive. With this extended definition, $y_j$ is a random variable that is $1$ or $-1$, each with probability $1/2$. In fact, these $y_j$ are independent. Indeed, for any $j$, the numbers $y_1,y_2,\ldots,y_{j-1}$ depend only on which part of $\mathscr{Q}_{j-1}$ is a superset of $Y$, but if we fix the part of $\mathscr{Q}_{j-1}$ which is a superset of $Y$, then $y_j$ is still $1$ or $-1$ with probability $1/2$. Let $S_j=y_1+y_2+\cdots+y_j$. The pair will be inactive if and only if for some $j \le i$, \[|S_j|\ge \frac{1}{2 \alpha}.\]

Now, we use the following well-known fact, which follows easily from the reflection principle: Let $t$ be a positive integer. Then

\[P( \exists j \le i: S_j \ge t)=P(S_i \ge t)+P(S_i \ge t+1). \]

Now, since the $y_i$ and hence the $S_i$ are symmetrically distributed and  $P(S_i \ge t+1) \leq P(S_i \ge t)$, we can conclude by the Chernoff bound that the probability that there is a $j \le i$ for which $|S_j| \ge \frac{1}{2\alpha}$ is at most \[4e^{-(\frac{1}{2\alpha})^2/(2i)} \le 4 e^{-1/(8\alpha^2i)} \le 4 e^{-1/(8\alpha^2s)}.\]

Thus, substituting in $s=1/36\alpha^{-2}$, then this will be at most $4e^{-1/(8/36)} \leq .05$. That is, in any step, at most a $.05$ proportion of the pairs of parts containing any given part will be inactive.
\end{proof}

Note that the conclusion of the above lemma is satisfied no matter what our choices are for the partitions in each step. In the following, we describe the properties that we want our graph to satisfy, and show that with positive probability, our construction will satisfy them. We will show that if the weighted graph $G$ constructed above has the desired properties, then it will give a construction which verifies Theorem \ref{main}. In fact, in Theorem \ref{maingen}, we will show that it has the stronger property that any vertex partition of $G$ with irregularity at most $\epsilon |V(G)|^2$ is not far from being a refinement of $\mathscr{P}_s$.

Fix an $i \ge 1$, a pair $X \in \mathscr{P}_{i-1}$, and $Y \in \mathscr{Q}_{i-1}$, and $a,b \in \{0,1\}$, such that the weight of $\widetilde G_i$ between $X_Y^a$ and $Y_X^b$ is not $0$ or $1$ (note that it is constant). In other words, any part of $\mathscr{P}_i$ which is a subset of $X_Y^a$ and any part of $\mathscr{Q}_i$ which is a subset of $Y_X^b$ form an active pair. The number of parts of $\mathscr{P}_i$ in $X_Y^a$ is $x_i$, and each such part is divided into 
$2x_{i+1}$ parts in $\mathscr{P}_{i+1}$. Similarly, the number of parts of $\mathscr{Q}_i$ in $Y_X^b$ is $x_i$, and each such part is divided into $2x_{i+1}$ parts in $\mathscr{Q}_{i+1}$. Fix a part $B$ of $\mathscr{P}_{i}$ with $B\subset X_Y^a$. Any part $C$ of $\mathscr{Q}_i$ with $C \subset Y_X^b$ will split $B$ into two collections of parts of $\mathscr{P}_{i+1}$, where both collections have size $x_{i+1}$. Now we want the bipartitions $B=B_C^0 \cup B_C^1$ of $B$ to satisfy the following two properties (see Figure 2).

\begin{enumerate}
\item Given two parts $C$ and $C'$ of $\mathscr{Q}_i$ that are subsets of $Y_X^b$, we obtain two different partitions of $B$, $B_{C}^0 \cup B_{C}^1$ and $B_{C'}^0 \cup B_{C'}^1$. Let $z_{hj}$ be the number of parts of $\mathscr{Q}_{i+1}$ in $B_{C}^h \cap B_{C'}^j$. Since $C$ and $C'$ both divide $B$ into two equal parts, we can see that $z_{00}=z_{11}$ and $z_{10}=z_{01}$. Also, $z_{00}+z_{01}=x_{i+1}$. Let $z=z(C,C')=z_B(C,C')=z_{00}-z_{01}$. We want every pair $C$ and $C'$ to satisfy $|z(C,C')| \le r:= \sqrt{6x_{i+1} \ln x_i}$. 

\item For any two vertices $u$ and $v$ in $B$ in different parts of $\mathscr{P}_{i+1}$, we say that a part $C$ \emph{separates} $u$ and $v$ if $u$ and $v$ lie in different parts in the partition $B=B_C^0 \cup B_C^1$. Let $y(u,v)$ be the number of parts of $\mathscr{Q}_i$ in $Y_X^b$ that do not separate $u$ and $v$ minus the number of parts of $\mathscr{Q}_i$ in $Y_X^b$ that do separate $u$ and $v$. Then, for any $v,u$, we want $|y(v,u)|\le t: = x_i/2$. \end{enumerate}

\begin{figure}[h]
\centering
\begin{tikzpicture}
\draw (0,0) ellipse (2 and 5);
\node at (0,5.5) {\LARGE $X \in \mathscr{P}_{i-1}$};
\draw (8,0) ellipse (2 and 5);
\node at (8,5.5) {\LARGE $Y \in \mathscr{Q}_{i-1}$};
\draw (0,1.5) circle (1.2);
\node at (0,3) {\LARGE $B \in \mathscr{P}_i$};
\draw (8,3) circle (1.2);
\node at (8,3) {\LARGE $C \in \mathscr{Q}_i$};
\draw (8,0) circle (1.2);
\node at (8,0) {\LARGE $C' \in \mathscr{Q}_i$};
\draw (-1.2,1.5) -- (1.2,1.5);
\node at (-1.5,1.8) {$B_C^1$};
\node at (-1.5,1.1) {$B_C^0$};
\draw (-.3,2.65) -- (.3,.35);
\node at (.75,.15) {$B_{C'}^1$};
\node at (.1,0) {$B_{C'}^0$};
\fill (-.5,1) circle (1pt) node[right] {$u$};
\fill (-.7,1.8) circle (1pt) node[right] {$v$};

\end{tikzpicture}

\caption{Part $B$ is divided into two parts by both $C$ and $C'$, the horizontal line shows how $C$ divides it, the other line how $C'$ divides it. We can see that $C$ separates $u$ and $v$, but $C'$ does not.}
\end{figure}
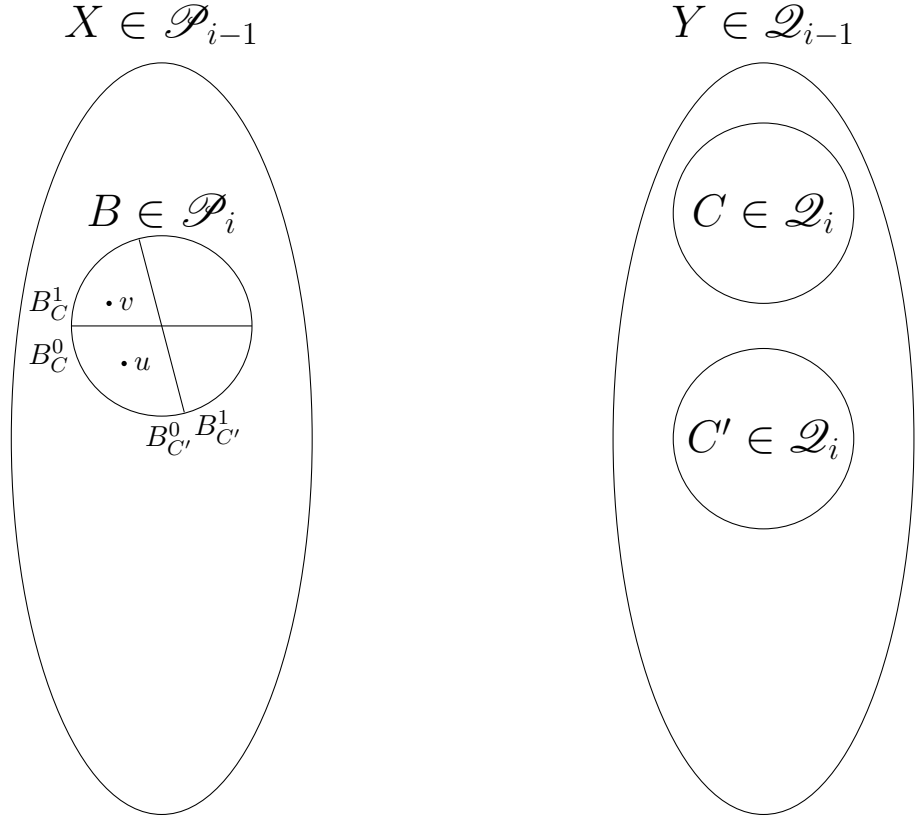

We will show that for each of these properties, the probability of failure, given a random set of $x_i$ partitions of $B$, is less than $1/2$, implying that the probability that neither of them fails is positive. The partitions are independent for different choices of $B$, so with positive probability, the two properties are satisfied for each $B \in \mathscr{P}_{i-1}$, $B \subset X_Y^a$. We require the analogous conditions for the bipartitions of each part in $\mathscr{Q}_i$ in $Y_X^b$. For the sake of brevity, we omit explicitly stating  these conditions for the other side. By symmetry between the two parts in the bipartite graph, the analogous conditions also hold with positive probability on the other side, and since the partitions on the two sides are independent, with positive probability they hold simultaneously for both sides. The conditions are also independent if we take a different pair $X,Y$ and $a,b$, or look at a different step $i$. This clearly implies that there exists a set of good partitions, that is, a set of partitions such that these two properties are satisfied for each $i$, for any of the required pairs $X \in \mathscr{P}_{i-1},Y \in \mathscr{Q}_{i-1}$, $a,b \in \{0,1\}$, such that the pairs of parts in step $i$ between $X_Y^a$ and $Y_X^b$ are active.

Thus, assume $i$ is fixed, $X \in \mathscr{P}_{i-1}$, $Y \in \mathscr{Q}_{i-1}$, and $a,b \in \{0,1\}$ such that any pair of parts of step $i$ between $X_Y^a$ and $Y_X^b$ is active, and $B \in \mathscr{P}_{i-1}$ with $B \subset X_Y^a$ is fixed. For fixed $C$ and $C'$, the value of $z_B(C,C')$ follows a hypergeometric distribution. As the hypergeometric distribution is at least as concentrated as the corresponding binomial distribution (for a proof, see Section 6 of \cite{H63}), we can apply the Chernoff bound to obtain
\[P(|z(C,C')|>r) \le 2e^{-r^2/(2x_{i+1})}.\]
Recall that $r= \sqrt{6x_{i+1} \ln x_i} $. Since the number of pairs is $x_i(x_i-1)/2$, the probability of the first property failing is at most 
\[\left(x_i(x_i-1)/2\right)2e^{-r^2/(2x_{i+1})} = x_i(x_i-1)e^{-3\ln x_i}= \frac{x_i-1}{x_i^2}<1/2.\] 

Now we show that the second property holds with probability of failure less than $1/2$. For a fixed pair $u$ and $v$, each bipartition separates them with probability $\frac{x_{i+1}}{2x_{i+1}-1}=\frac{1}{2}+\frac{1}{4x_{i+1}-2}$, independently of the other bipartitions. 

For a real number $p$ and positive integer $n$, let $B_{n,p}$ denote the binomial distribution with $n$ trials and probability value $p$, so if $Z$ is a random variable with distribution $B_{n,p}$, then $P(Z=k)={n \choose k}p^k(1-p)^{n-k}$. If $Z_1$ is the random variable with distribution $B_{n,\frac{1}{2}}$ and $Z_2$ is the random variable with distribution $B_{n,\frac{1}{2}+\delta}$ with $|\delta|<\frac{\ln 2}{2n}$, then $P(Z_2=k)\leq 2P(Z_1=k)$ for each $k$. Indeed,
\[\frac{P(Z_2=k)}{P(Z_1=k)}=(1+2\delta)^k(1-2\delta)^{n-k} \le (1+2 |\delta|)^n \le e^{2 |\delta| n} < 2.\] 
Note that we can write $y(u,v)$ as $2x(u,v)-n$, where $x(u,v) \sim B_{n,p}$, with $n=x_i$ and $p=\frac{1}{2}+\delta$ with $\delta=-\frac{1}{4x_{i+1}-2}$. Since $x_{i+1} = 2^{x_i/16}$, and $x_i \ge x_1=2^{10}$, it is easy to check that $|\delta|=\frac{1}{4x_{i+1}-2}<\frac{\ln 2}{2x_i}=\frac{\ln 2}{2n}$. If we instead take $y'(u,v)=2x'(u,v)-n$ with $x'(u,v) \sim B_{n,1/2}$, by the Chernoff bound, we have 
\[P\left(\left |y'(v,u)\right | > t\right) < 2e^{-t^2/(2x_i)}.\]

By the above analysis of binomial distributions with probability $p$ near $1/2$, we thus have
\[P\left(\left |y(v,u)\right | > t\right) < 2 \cdot 2e^{-t^2/(2x_i)}.\]

Recall that $t=x_i/2$, and $x_{i+1}=2^{x_i/16}$. We have ${2x_{i+1} \choose 2}<2x_{i+1}^2$ pairs of parts of $\mathscr{P}_{i+1}$ which are subsets of $B$. Thus, since $x_i \ge x_1=2^{10}$, the probability of the second property failing is less than
\[2x_{i+1}^2 4e^{-t^2/(2x_i)} = \exp( (\ln 2)(3+x_i/8)-x_i/8)= \exp( x_i(\ln 2  - 1)/8 + 3\ln 2) < 1/2. \]

Thus, we indeed have a probability of failure less than $1/2$ for each of the two properties. For the remainder of the proof we will fix such a graph $G$ with these properties for each $i$, for any pair $X \in \mathscr{P}_{i-1},Y \in \mathscr{Q}_{i-1}$, any $a,b \in \{0,1\}$, and any $B \in \mathscr{P}_{i}$ with $B \subset X_Y^a$, and the same conditions holds when we switch the two parts. We will show that the irregularity is at least $\epsilon|V(G)|^2$ in any partition with not too many parts.  

\begin{theorem}\label{maingen} 
The weighted graph $G$ constructed in the above manner has the property that any partition with less than $.97(|\mathscr{P}_s|+|\mathscr{Q}_s|)/x_1$ parts has irregularity greater than $\epsilon |V(G)|^2$.
\end{theorem}

We will show that we can assume that our partition is a refinement of the first partition, $\mathscr{P}_1 \cup \mathscr{Q}_1$. This will give us a factor of $32x_1^2$ in the irregularity, and $4x_1$ in the size of the partition. Assuming this, we will show something a bit stronger. We will show that if the partition has irregularity at most $32x_1^2 \epsilon |V(G)|^2$, then the partition is not too far from being a refinement of the joint partition $\mathscr{P}_s \cup \mathscr{Q}_s$, which implies that it has at least half as many parts.

\section{Singular Values and Edge Distribution in Matrices} \label{sing}

In this section, we prove some properties about the edge distribution of bipartite graphs. The key lemma gives a bipartite analogue of the expander mixing lemma of Alon and Chung \cite{AC88}. While it is a modification of a standard proof, we give the proof here for completeness. 

If we have any weighted bipartite graph (with potentially negative weights) between two sets $X$ and $Y$, each of size $n$, we can represent it by an $n \times n$ matrix $A$, with $X$ corresponding to the rows, $Y$ corresponding to the columns, and the entry value equal to the weight of the corresponding edge. Let $A=O\Lambda U^T$ be the singular value decomposition. Thus, $\Lambda = \diag(\sigma_1,\sigma_2,\ldots,\sigma_n)$, $\sigma_i \ge 0$, and $O$ and $U$ are orthogonal real matrices (recall that $A$ has real entries). Let $\lambda=\max_{i} \sigma_i$. Also, for any $v \in X$, $C \subset Y$, define 
\[d_C(v)=\sum_{w \in C} a_{vw}.\] This is the sum of the edge weights between $v$ and the vertices in $C$. We next observe that if $\lambda$ is small, then the graph between $X$ and $Y$ is very uniform. The following lemma shows that for any subset of the vertices on one side, most of the vertices on the other side will have density about equal to the average density, which in our case is zero. This will imply that the number of edges between any two sets is small in absolute value compared to their sizes.

\begin{lemma} \label{expmix}
Suppose $A$ is a matrix representing a weighted graph $G$ between $X$ and $Y$, and suppose $\lambda$ is the maximum of the singular values $\sigma_i$. We have the following.

\begin{enumerate}

\item For any $C \subset Y$,
\[ \sum_{v \in X} d_C(v)^2 \le \lambda^2 |C| .\]

\item For any $B \subset X$ and $C \subset Y$,
\[|e(B,C)| \le \lambda \sqrt{|B||C|} .\]

\end{enumerate}

\end{lemma}
\begin{proof}
We first prove part 1. Let $\bold{y}$ be the vector with entries corresponding to elements of $Y$, with $1$ at the elements of $C$, and $0$ at the other elements. We claim that
\[\langle A\bold{y},A \bold{y} \rangle \le \lambda^2 \langle \bold{y},\bold{y} \rangle.\]
Indeed, recall that $A=O \Lambda U^T$, and if we let $\bold{y}'=U^T\bold{y}$, then 
\begin{equation}\label{eq:add1}\langle A\bold{y},A \bold{y} \rangle= \langle O \Lambda U^T\bold{y},O \Lambda U^T \bold{y} \rangle={\bold{y}'}^T\Lambda^2\bold{y}' \le \lambda^2 \langle \bold{y}',\bold{y}' \rangle = \lambda^2 \langle \bold{y},\bold{y} \rangle. \end{equation}

The first equality in (\ref{eq:add1}) follows by substituting for $A$, the second by substituting $y'$ and using the fact that $O$ is orthogonal. The inequality follows from the fact that $\Lambda$ is a diagonal matrix containing the singular values, and the last equality follows from the fact that $U$ is orthogonal.

However, we know that 
\[A\bold{y}=(d_C(v),v \in X),\]
and thus  
\begin{equation}\label{eq:add2}\langle A\bold{y},A\bold{y} \rangle =\sum_{v \in X} d_C(v)^2.\end{equation}
However, it is easy to see that \begin{equation}\label{eq:add3}\langle \bold{y},\bold{y} \rangle =|C|.\end{equation} Combining (\ref{eq:add1}), (\ref{eq:add2}), and (\ref{eq:add3}), we obtain
\[ \sum_{v \in X} d_C(v)^2 \le \lambda^2 |C|.\]
This completes the proof of part 1.

For part 2, we use the Cauchy-Schwarz inequality, and the statement of part 1. We can see that 
\[|e(B,C)| =|\sum_{v \in B} d_C(v)| \le \sqrt{|B|} \sqrt{\sum_{v \in B} d_C(v)^2} \le \sqrt{|B|} \sqrt{\sum_{v \in X} d_C(v)^2} \le \lambda \sqrt{|B||C|}.\]
This completes the proof of part 2.
\end{proof}

Suppose we have an $n \times n$ matrix $A$ with singular values $\sigma_1,\sigma_2,\ldots,\sigma_n$. Consider the {\it $k$-blow-up} $A'$ of $A$, which is a $kn \times kn$ matrix where we replace each entry $a_{ij}$ of $A$ with the $k \times k$ matrix $a_{ij}E$, 
where $E$ is the constant matrix with $1$'s 
everywhere. In other words, $A'=A \otimes E$ is the tensor product (also known as the Kronecker product) of $A$ and $E$. Note that $A'$ corresponds to the matrix we obtain by replacing each vertex by $k$ (independent) vertices. We have the following lemma.

\begin{lemma} \label{blowup}
Let $A'$ be the $k$-blow-up of an $n \times n$ matrix $A$, and $\lambda$ and $\lambda'$ be the largest singular values of $A$ and $A'$ respectively. Then $\lambda'=k\lambda$.
\end{lemma}

\begin{proof}

We will show that if the singular values of $A$ are $\sigma_1,\ldots,\sigma_n$, then the singular values of $A'$  are $k \sigma_1,k \sigma_2,\ldots,k 
\sigma_n,0, 0,\ldots$. That is, we have each 
original singular value multiplied by $k$ 
once, and the remaining singular values of $A'$ are zeroes. This clearly implies the lemma. Let $O_k$ be a $k \times k$ 
matrix that is orthogonal, and the first 
column consists of $\frac{1}{\sqrt{k}}$ 
everywhere (such an $O_k$ clearly exists). Then $O_k^TEO_k=\diag(k,0,0,\ldots,)=kE_{11}$, thus,
 the singular values of $E$ are $k,0, \ldots, 0$. If $\Lambda=\diag(\sigma_1,\sigma_2,\ldots,\sigma_n)=U^TAV$,
then set $\widetilde{U}=U \otimes O_k$ and $\widetilde{V}=V \otimes O_k$. As the tensor product of orthogonal matrices is orthogonal, the matrices $\widetilde{U}$ and $\widetilde{V}$ are orthogonal matrices. 

\begin{multline*} \widetilde{U}^T A' \widetilde{V} = (U^T \otimes O_k^T)(A \otimes E)(V \otimes O_k) =(U^TAV)\otimes(O_k^TEO_k)=\Lambda \otimes kE_{11}\\
= \diag(k\sigma_1,0,\ldots,0,k\sigma_2,0,\ldots,0,k\sigma_3,0,\ldots,0,k\sigma_4,\ldots,) .\end{multline*}
The second equality above is by the mixed-product property of the tensor product. Hence, the singular values are as claimed, which completes the proof. 

\end{proof}

\section{Irregularity Between Parts} \label{graphchangesec}

Recall that a part of step $i$ refers to a part of $\mathscr{P}_i$ or $\mathscr{Q}_i$. Next, we establish the fact that while $G$ is quite irregular between any two parts $X$ and $Y$ of step $i-1$ (on opposite sides), almost all of this irregularity comes from $\widetilde{G}_i$. In fact, almost all of the irregularity comes from $G_i$, since the $G_j$ with $j<i$ are constant across $X$ and $Y$. Concretely, it is easy to check that $\irreg_{\widetilde{G}_i}(X,Y)=\frac{\alpha}{4}|X||Y|$, coming from considering the subsets $X_Y^0 \subset X$ and $Y_X^0 \subset Y$. In contrast, Lemma \ref{regular} below easily implies that $\irreg_{G-\widetilde{G}_i}(X,Y) = O(x_i^{-1/4}\alpha |X||Y|)$. The following lemma is key in establishing this useful fact. The reader may ask why, if this is true for $i=1$, we don't just take a four element partition: $\{V_W^0,V_W^1,W_V^0,W_V^1\}$. The reason is that in this case the irregularity is larger than $\epsilon |V||W|$. The bound we get from the lemma will be useful for $i \geq 2$. 

\begin{lemma} 
If $X \in \mathscr{P}_{i-1}$ and $Y \in \mathscr{Q}_{i-1}$ are two parts in step $i-1$, $a,b \in \{0,1\}$, then for any $U \subset X_Y^a$ and $Z \subset Y_X^b$, we have $|e_{G_{i+1}}(U,Z)| \le .8 x_i^{-1/4} \alpha n\sqrt{|U||Z|} \le .8 x_i^{-1/4}\alpha n^2$, where $n=|X_Y^a|=|Y_X^b|$.
\end{lemma}

\begin{proof}

There are three cases to consider. The first case is that the pair $X,Y$ is already inactive. In this case, all the weights in $G_{i+1}$ between $X$ and $Y$ are zero (recall that if a pair of parts becomes inactive, then in later steps any pair of parts coming from this pair is also inactive). The second case is that the pair $X,Y$ is active, but after adding the weights in $G_i$, any edge between $X_Y^a$ and $Y_X^b$ will have weight $0$ or $1$. In this case, for any two parts $X' \subset X_Y^a$ and $Y' \subset Y_X^b$ of step $i+1$, the pair $X',Y'$ will be inactive, thus, in $G_{i+1}$, all edges between them will again have weight zero.

In the remaining case, any pair of parts of step $i+1$ coming from $X_Y^a$ and $Y_X^b$ is active.

Let $A'$ be the $n \times n$ matrix corresponding to the weighted graph $G_{i+1}$ restricted to $X_Y^a$ and $Y_X^b$. We will show that  the largest singular value $\lambda'$ of $A'$ is at most $.8 x_i^{-1/4} \alpha n$. This implies the desired result by Lemma \ref{expmix}. 

Recall that the edge weights in $G_{i+1}$ are equal between each part in $\mathscr{P}_{i+1}$ and each part in $\mathscr{Q}_{i+1}$. Let $H_{i+1}$ be the bipartite graph with vertex sets  $\mathscr{P}_{i+1}$ and $\mathscr{Q}_{i+1}$, so each part of step $i+1$ is a vertex, and each edge in $H_{i+1}$ has weight equal to the density in $G_{i+1}$ across that pair of parts. Thus, $G_{i+1}$ is a blow-up of $H_{i+1}$. Let $\widetilde{X}$ and $\widetilde{Y}$ be the sets of vertices of $H_{i+1}$ that consists of those parts of step $i+1$ that come from $X_Y^a$ and $Y_X^b$, respectively. Let $A$ be the matrix corresponding to the weighted graph $H_{i+1}$, restricted to $\widetilde{X}$ and $\widetilde{Y}$, and $\lambda$ be the largest singular value of $A$. Then, $A$ is an $m \times m$ matrix, where $m=2x_ix_{i+1}$, the number of vertices in each of $\widetilde{X}$ and $\widetilde{Y}$. By Lemma \ref{blowup}, we have $\lambda'=\lambda\frac{n}{m}$, thus, it suffices to show that $\lambda \leq .8 x_i^{-1/4} \alpha m$.

Consider the matrix $M=AA^T$. Note that \[\tr  M^2= \tr(AA^TAA^T) = \sum_{i=1}^m \sigma_i^4 \geq \max \sigma_i^4 = \lambda^4.\]
To complete the lemma, it thus suffices to show that $\tr M^2 \leq (.8\alpha m)^4/x_i$.

The rows and columns of $M$ both correspond to the parts of $\mathscr{P}_{i+1}$ that are subsets of $X_Y^a$. Any entry of $M$ in the diagonal will be $m\alpha^2$, but the other entries will be much smaller.

First, consider an element off the diagonal, corresponding to $v$ and $u$ in $\widetilde{X}$ that are in the same part $B$ of $\mathscr{P}_{i}$. Then for any vertex $w$ in any part that does not separate them,  $A_{vw}=A_{uw}=\pm \alpha$, and for a vertex in a part that does separate them, $A_{vw}=-A_{uw}=\pm \alpha$. Thus, 
\[M_{vu}=\sum_{w \in \widetilde{Y}} A_{vw}A_{uw}=y(v,u)2x_{i+1}\alpha^2,\]
where $y(u,v)=y_B(u,v)$ is as defined at the end of Section \ref{construction}, which is the difference between the number of parts that do not separate them and the number of parts that do separate them. By construction, $y(u,v) \leq t= x_i/2$, and hence \[|M_{vu}| \le 2tx_{i+1}\alpha^2.\]

Now, if $v$ and $u$ come from two different parts $B$ and $B'$ of $\mathscr{P}_{i}$, then for any part $C \subset Y_X^b$ of $\mathscr{Q}_i$, it is divided ``almost evenly'' by $B$ and $B'$, that is, about a quarter of the vertices in $C$ have an edge with $+\alpha$ going to both $v$ and $u$, a quarter with $+\alpha$ to $v$, $-\alpha$ to $u$, a quarter with $-\alpha$ to $v$, $+\alpha$ to $u$, and a quarter with $-\alpha$ to both. The difference is measured by $z_C(B,B')$ as defined in Section \ref{construction}. To be precise,
\[M_{vu}=\sum_{w \in \widetilde{Y}} A_{vw}A_{uw}=\sum_{C \in \mathscr{Q}_{i+1} \subset Y_X^b} 2z_C(B,B')\alpha^2.\]

Thus, 
\[|M_{vu}| \le \sum_{C \in \mathscr{Q}_{i+1} \subset Y_X^b} |2z_C(B,B')|\alpha^2 \le 2x_i r \alpha^2,\]
where we recall that $r=\sqrt{6x_{i+1}\ln x_i}$.

There are $m=2x_ix_{i+1}$ diagonal entries $M_{vv}$ in $M$, $x_i(2x_{i+1})(2x_{i+1}-1)$ entries $M_{vu}$ in the same part of $\mathscr{P}_i$ with $u \not = v$, and $x_i(x_i-1)(2x_{i+1})^2$ entries $M_{vu}$ with $u$ and $v$ in different parts of $\mathscr{P}_i$. Hence, 
\begin{align*} \tr M^2 & = \sum_{v,u \in X_Y^a} M_{vu}^2 \le m (m \alpha^2 )^2+x_i\left(2x_{i+1}(2x_{i+1}-1)\right)4t^2x_{i+1}^2\alpha^4 + x_i(x_i-1)(2x_{i+1})^24x_i^2 r^2 \alpha^4 \\ 
& \le  m^3\alpha^4+4x_i^3x_{i+1}^4 \alpha^4+ 96x_i^4 x_{i+1}^3 \ln x_i \alpha^4 = m^4\alpha^4\left(\frac{1}{m}+ \frac{1}{4x_i}+\frac{6 \ln x_i}{x_{i+1}}\right)  \le \frac{(.8\alpha m)^4}{x_i}. 
\end{align*}

The last inequality is true because, if we recall that $m=2x_ix_{i+1}$, $x_{i+1}=2^{x_i/16}$, we have
\[\frac{1}{2x_{i+1}}+\frac{1}{4}+\frac{6x_i \ln x_i}{x_{i+1}} \le .4 \le (.8)^4. \]
(The expression on the left is maximized if $i=1$, in which case $x_i = x_1 = 2^{10}$, $x_{i+1} =x_2= 2^{2^6}$.)

This estimate completes the proof.

\end{proof}

Note that if $j \ge i$, we get the following corollary of the previous lemma, since we can divide $X$ and $Y$ into the parts of $\mathscr{P}_{j-1}$ and $\mathscr{Q}_{j-1}$ that are contained in them, and sum over all pairs. 

\begin{corollary} 
Let $X \in \mathscr{P}_{i-1},Y \in \mathscr{Q}_{i-1}$, and $a,b \in \{0,1\}$. If $U \subset X_Y^a$, $Z \subset Y_X^b$, and $j \geq i$, then  $|e_{G_{j+1}}(U,Z)| \le  .8 x_j^{-1/4}\alpha |X_Y^a||Y_X^b|$.
\end{corollary}

Now, from this, it is easy to obtain the following lemma, which is the main result in this section, and an important tool in the proof of the main result. 

\begin{lemma} \label{regular}
Let $X \in \mathscr{P}_{i-1}$, $Y \in \mathscr{Q}_{i-1}$, and $a,b \in \{0,1\}$. If $U \subset X_Y^a$ and $Z \subset Y_X^b$, then 
\[|e_{G-\widetilde{G}_i}(U,Z)| \le .9 x_i^{-1/4} \alpha |X_Y^a||Y_X^b|.\]
\end{lemma}
\begin{proof}
Using the triangle inequality, and applying the previous corollary for each $j$ from $i$ to $s-1$, we have
\[|e_{G-\widetilde{G}_i}(U,Z)|=\left|\sum_{j=i}^{s-1} e_{G_{j+1}}(U,Z)\right| \le \sum_{j=i}^{s-1} \left|e_{G_{j+1}}(U,Z)\right| \le \sum_{j=i}^{s-1} .8 x_j^{-1/4} \alpha n^2 \le .9 x_i^{-1/4} \alpha |X_Y^a||Y_X^b|.
\]

Recall that $x_{j+1}=2^{x_j/16}$ and $x_1=2^{10}$, thus the sequence $(x_j)_{j \geq i}$ grows very rapidly, which justifies the last inequality.
\end{proof}

\section{Proof of the Lower Bound} \label{secmain} 

The goal of this section is to prove Theorem \ref{main}, by first proving Theorem \ref{maingen}. We therefore suppose we have a partition of the vertex set of $G$. We will show that if this partition is far from being a refinement of $\mathscr{P}_s \cup \mathscr{Q}_s$, then it has large irregularity. We first assume that this partition is a refinement of the bipartition, that is, we have two partitions, $\mathscr{S}$ of $V$ and $\mathscr{T}$ of $W$. In fact, we will further assume for now that $\mathscr{S}$ is a refinement of $\mathscr{P}_1$, and $\mathscr{T}$ is a refinement of $\mathscr{Q}_1$. We will later show (see Lemma \ref{divide}) how to get rid of these assumptions. Our goal for now is to prove the following lemma. 

\begin{lemma}\label{withassumption}
Assume $\mathscr{S}$ is a refinement of $\mathscr{P}_1$ and $\mathscr{T}$ is a refinement of $\mathscr{Q}_1$. If the vertex partition $\mathscr{S} \cup \mathscr{T}$ has irregularity at most $\frac{1}{5000}\alpha|V||W|$, then at least a $.97$ proportion of the vertices in $V$ are in parts in $\mathscr{S}$ such that more than half of its vertices are in the same part of $\mathscr{P}_s$, and at least a $.97$ proportion of the vertices in $W$ are in parts in $\mathscr{T}$ such that more than half of its vertices are in the same part of $\mathscr{Q}_s$. It follows that the number of parts in $\mathscr{S} \cup \mathscr{T}$ is at least $.97(|\mathscr{P}_s|+|\mathscr{Q}_s|)$. 
\end{lemma}

First, we will establish a simple fact that makes future calculations easier. It uses the triangle inequality to show that the irregularity between a pair $S$, $T$ of parts is large if there are large subsets $S', S'' \subset S$ and $T', T'' \subset T$ for which the densities $d(S',T')$ and $d(S'',T'')$ are not close. 

\begin{lemma} \label{irreghelpsinglepair} 
Suppose we have, for two sets of vertices $S$ and $T$, subsets $S', S'' \subset S$ and $T' , T'' \subset T$ for which $|S''||T''| \le |S'||T'|$. Then

\[\irreg(S,T) \ge \frac{1}{2} \left|e_G(S'',T'')-\frac{|S''||T''|}{|S'||T'|}e_G(S',T')\right | .\]
\end{lemma}

\begin{proof}
By the triangle inequality, we have that either

\begin{equation}\label{eq1or2a}\left|e_G(S'',T'')-\frac{|S''||T''|}{|S||T|}e_G(S,T)\right| \ge \frac{1}{2} \left |e_G(S'',T'')-\frac{|S''||T''|}{|S'||T'|}e_G(S',T')\right |,\end{equation}

or 

\begin{equation}\label{eq1or2b}\left |\frac{|S''||T''|}{|S'||T'|}e_G(S',T')-\frac{|S''||T''|}{|S||T|}e_G(S,T)\right | \ge \frac{1}{2} \left |e_G(S'',T'')-\frac{|S''||T''|}{|S'||T'|}e_G(S',T')\right |.\end{equation}
Indeed, the sum of the left hand sides of (\ref{eq1or2a}) and (\ref{eq1or2b}) is at least the sum of the right hand sides. If (\ref{eq1or2a}) holds, then we get the desired bound on the irregularity using the subsets $S''$ and $T''$. So we may suppose (\ref{eq1or2b}) holds. 

Since $|S''||T''| \le |S'||T'|$, multiplying the left side of (\ref{eq1or2b})  by $\frac{|S'||T'|}{|S''||T''|}\ge 1$ gives that

\[\left|e_G(S',T')-\frac{|S'||T'|}{|S||T|}e_G(S,T)\right| \ge \frac{1}{2} \left |e_G(S'',T'')-\frac{|S''||T''|}{|S'||T'|}e_G(S',T')\right |.\]

Since $\irreg(S,T)$ is the maximum of the expression on the left over all pairs of subsets of $S$ and $T$, this gives that
\[ \irreg(S,T) \ge \frac{1}{2} \left |e_G(S'',T'')-\frac{|S''||T''|}{|S'||T'|}e_G(S',T')\right |.\]

\end{proof}

Our goal is to give a lower bound on the irregularity of the partition $\mathscr{S} \cup \mathscr{T}$. Since $G$ is bipartite, the irregularity is zero between a pair of parts in $\mathscr{S}$ or a pair of parts in $\mathscr{T}$. Hence, $\irreg(\mathscr{S} \cup \mathscr{T})$ is equal to \[\irreg(\mathscr{S},\mathscr{T}):=\sum_{\substack{S \in \mathscr{S}\\T \in \mathscr{T}}}\irreg(S,T).\] 
If we have such subsets for the previous lemma for certain pairs of parts, by collecting the irregularity between these pairs, we obtain the following corollary giving a lower bound on the irregularity of the partition. 

\begin{corollary} \label{irreghelp}

Suppose we have an $\mathscr{R} \subset \mathscr{S} \times \mathscr{T}$, and, for each pair of parts $(S,T) \in \mathscr{R}$, subsets $S_T^0, S_T^1 \subset S$ and $T_S^0, T_S^1 \subset T$ satisfying $|S_T^0||T_S^0| \le |S_T^1||T_S^1|$. Then, 

\[\irreg(\mathscr{S},\mathscr{T}) \ge \frac{1}{2} \sum_{(S,T) \in \mathscr{R}} \left | e_G(S_T^0,T_S^0)-\frac{|S_T^0||T_S^0|}{|S_T^1||T_S^1|}e_G(S_T^1,T_S^1)\right | .\]

\end{corollary}

\begin{proof}

By the previous lemma, for any pair of $(S,T) \in \mathscr{R}$, we have

\[\irreg(S,T) \ge \frac{1}{2} \left |e_G(S_T^0,T_S^0)-\frac{|S_T^0||T_S^0|}{|S_T^1||T_S^1|}e_G(S_T^1,T_S^1)\right | .\]

Adding this up for all $(S,T) \in \mathscr{R}$, we obtain

\[\irreg(\mathscr{S},\mathscr{T}) \ge \frac{1}{2} \sum_{(S,T) \in \mathscr{R}} \left | e_G(S_T^0,T_S^0)-\frac{|S_T^0||T_S^0|}{|S_T^1||T_S^1|}e_G(S_T^1,T_S^1)\right |, \]

which is what we wanted to show.

\end{proof}

For each step $i$, we define a coloring of the vertices. For a vertex $v$, look at the element $S$ in $\mathscr{S}$ or $\mathscr{T}$ containing $v$, and the part $X$ in step $i$ containing $v$. If $|S \cap X| > |S|/2$, color $v$ blue, otherwise, color it red. 

Since we keep refining partitions, a vertex can change from blue to red, but it can never change back to blue. By assumption, $\mathscr{S}$ refines $\mathscr{P}_1$ and  $\mathscr{T}$ refines $\mathscr{Q}_1$, so every vertex is blue in step $1$.

For each $i$, let $R_i$ be the set of vertices that turn red in step $i$. Let $\widetilde{R}_i$ be the set of vertices that are red in step $i$. Thus, $\widetilde{R}_i=R_1 \cup R_2 \cup R_3 \cup \ldots \cup R_i$ is a partition of $\widetilde{R}_i$. We have  $R_1 = \emptyset$ as every vertex is blue in step $1$. For any part $S$ of $\mathscr{S} \cup \mathscr{T}$, let $i_S$ be the last step when it contains blue vertices. Note that more than half of the vertices of $S$ are blue in step $i_S$, but in step $i_S+1$ (and later steps), there are no blue vertices. For $j \le i_S$, let $S_j$ be the set of vertices in $S$ that are blue in step $j$. Let $\widetilde{S}_j$ be the set of vertices in $S$ that were blue in the previous step, but are red now, that is, $\widetilde{S}_j=S_{j-1} \setminus S_j$. Also, for $S \in \mathscr{S}$ and each $j \leq i_S$ there is a unique $X \in \mathscr{P}_j$ such that $S_j \subset X$. For each $Y \in \mathscr{Q}_j$, this $X$ is divided into $X_Y^0$ and $X_Y^1$, and for $a=0,1$, let $S_{j,Y}^a=S_{j} \cap X_Y^a$. For each $S$ and $Y$, let $S_{j,Y}$ be the smaller and $S_{j,Y}^*$ be the larger of $S_{j,Y}^0$ and $S_{j,Y}^1$, breaking ties arbitrarily. 



The next lemma provides an important estimate in establishing a lower bound on the irregularity. The set-up is that we have a fixed $S \in \mathscr{S}$ and $X \subset \mathscr{P}_{j-1}$ such that $S_{j-1} \subset X$. That is, $|S_{j-1}|=|S \cap X| \geq |S|/2$. It then roughly says that $S_{j,Y}$ is typically a large fraction of how much $S_{j-1}$ breaks off at the next step. Note that in the next lemma the averaging for a typical $Y$ is  weighted by $|Y \cap W'|$, where $W'$ is a large subset of $W$. For a vector $\lambda=(\lambda_i) \in \mathbb{R}^k$ and $1 \leq p < \infty$, we write $\| \lambda \| _p$ for $(\sum_{i=1}^k |\lambda_i|^p)^{1/p}$ and $\| \lambda \|_{\infty}$ for $\max_{1 \leq i \leq k} |\lambda_i|$. For $X \in \mathscr{P}_{j-1}$, let $\mathscr{Q}_X$ be the set of $Y \in \mathscr{Q}_{j-1}$ such that $X,Y$ is an active pair.

\begin{lemma} \label{sumalot}
Let $2 \le j \le s$. Let $W' \subset W$ with $|W'|=C|W|$. For each $Y \in \mathscr{Q}_{j-1}$, let $Y'=Y \cap W'$.  Let $S \in \mathscr{S}$ be such that $i_S \ge j-1$, and let $X \in \mathscr{P}_{j-1}$ be the unique part with $S_{j-1} \subset X$. 
\begin{enumerate}
\item If $i_S\ge j$, then
\[\sum_{Y \in \mathscr{Q}_X} |S_{j-1,Y}||Y'| \ge |\widetilde{S}_j|(C-.8)|W| . \]

\item If $i_S=j-1$, then
\[\sum_{Y \in \mathscr{Q}_X} |S_{j-1,Y}||Y'| \ge \frac{1}{2}(|S_{j-1}|-|S|/2)(C-.8)|W|. \]

\end{enumerate}

\end{lemma}

\begin{proof}

In the case $C \leq .8$, the desired bound is trivial, and so we may (and will) assume $C> .8$. 

First, note that by Lemma \ref{inactive}, at most a $.05$ proportion of the $Y$ are inactive with $X$, so at most $.05|W|$ vertices are in a $Y$ that is inactive with $X$. The set $S_{j-1}$ is divided into parts $A_1,\ldots,A_k$ by $\mathscr{P}_j$. First, look at a fixed pair of distinct parts $A_h$ and $A_i$. By construction, at most a $\frac{3}{4}$ proportion of the $Y$ in $\mathscr{Q}_{j-1}$ do not separate $A_h$ and $A_i$, and so at most $\frac{3}{4}|W|$ vertices are in a $Y$ that does not separate them. Call $Y \in \mathscr{Q}_{j-1}$ {\it good (with respect to the pair $h,i$)} if the pair $X,Y$ is active and $Y$ separates $A_h$ and $A_i$. For each $h$ and $i$, let $\mathscr{R}(h,i)$ be the set of $Y$ which are good. It follows that \begin{equation}\label{eq:2in} \sum_{Y \in \mathscr{R}(h,i)} |Y'| \geq |W'|-.05|W|-\frac{3}{4}|W| = \left(C-\frac{3}{4}-.05\right)|W|=\left(C-.8\right)|W|.\end{equation}

Now, we look at the first case. Since $i_S \ge j$, $S_j$ is well-defined, and is one of the $A_i$. Without loss of generality, we may assume $A_1=S_j$. Then $\widetilde{S}_j$ is the union of all $A_h$ with $h \ge 2$. Since $\left|S_j\right| \ge |S|/2$, for each $Y$, the partition of $S_{j-1}$ into two parts corresponding to $Y$ satisfies that $A_1$ is a subset of the larger part $S_{j-1,Y}^*$ and is hence disjoint from the smaller part $S_{j-1,Y}$. Thus $S_{j-1,Y}$ is the union of those $A_h$ for which $Y$ separates the vertices in $A_1$ from the vertices in $A_h$. Using \eqref{eq:2in}, we have
\begin{align*}\sum_{Y \in \mathscr{Q}_X} |S_{j-1,Y}||Y'|= \sum_{\substack{h \ge 2\\ Y \in \mathscr{R}(1,h)}} |A_h||Y'| \geq \sum_{h \ge 2} |A_h|\left(C-.8\right)|W| =  |\widetilde{S}_j|\left(C-.8\right)|W|.\end{align*} This is the desired inequality, and completes this case.

We next consider the second case. Let $\lambda_h=|A_h|$. For each $Y$ for which $X,Y$ is an active pair, we have 
\begin{equation}\label{eq:sys} |S_{j-1,Y}|\left(|S_{j-1}|-|S_{j-1,Y}|\right)=\sum_{\substack{h<i\\ Y \in \mathscr{R}(h,i)}}\lambda_{h}\lambda_{i}.\end{equation}
Thus, 
\begin{align*} \sum_{Y \in \mathscr{Q}_X} |S_{j-1,Y}| |S_{j-1}| |Y'| \geq & \sum_{Y \in \mathscr{Q}_X} |S_{j-1,Y}| \left(|S_{j-1}|-|S_{j-1,Y}|\right) |Y'| =
\sum_{Y \in \mathscr{Q}_X}\sum_{\substack{h<i\\ Y \in \mathscr{R}(h,i)}} \lambda_{h}\lambda_{i}|Y'| \\ = & \sum_{h<i} \lambda_{h}\lambda_{i}\sum_{Y \in \mathscr{R}(h,i)}|Y'| \geq  \sum_{h<i} \lambda_{h}\lambda_{i}\left(C-.8\right)|W|\\ = &\left(C-.8\right)|W| \sum_{h<i} \lambda_{h}\lambda_{i}= 
\left(C-.8\right)|W| \left(\| \lambda \|_1^2 - \| \lambda \|_2^2 \right)/2 \\ \geq & \left(C-.8\right)|W| \|\lambda \|_1  \left(\| \lambda \|_1 - \| \lambda \|_{\infty} \right)/2 \\ = & \left(C-.8\right)|W| |S_{j-1}| \left(|S_{j-1}| - \| \lambda \|_{\infty} \right)/2. \end{align*}
The first equation above is from substituting in (\ref{eq:sys}), the second equation is by changing the order of summation, the second inequality is by substituting in (\ref{eq:2in}), the last inequality uses the inequality $\| \lambda \|_2^2 \leq \| \lambda \|_{\infty}\| \lambda \|_1$, and the last equality is by substituting $\| \lambda \|_1=|S_{j-1}|$. Dividing both sides by $|S_{j-1}|$, we obtain (using that $\| \lambda \|_{\infty} \le |S|/2$)
\[ \sum_{Y \in \mathscr{Q}_X} |S_{j-1,Y}||Y'| \geq  \left(C-.8\right)|W| \left(|S_{j-1}| - \| \lambda \|_{\infty} \right)/2 \geq \left(C-.8\right)|W| \left(|S_{j-1}| - |S|/2\right)/2,\]
which completes the proof. 
\end{proof}

Our basic strategy is to show that if a significant proportion of the vertices turn red by step $s$, then we must have a certain amount of irregularity. There will be two approaches to show this. If the number of red vertices increases by a substantial amount in a single step, then we will look only at the partitions in that step, and analyzing these we will obtain enough irregularity. However, it is possible that the number of red vertices slowly increases on both sides, and eventually after many steps it adds up to a substantial amount by a certain step, even though it never increases by much in any single step. The following lemma will be applicable in this case. It obtains an amount of irregularity from each step proportional to the number of new red vertices $j \le i$. With careful bookkeeping, it manages to add these up to give a substantial lower bound on the irregularity.

\begin{lemma} \label{addingup}

Suppose in step $i \geq 2$ we have $C|W|$ blue vertices in $W$ and $r|V|$ red vertices in $V$. Then   \[\irreg(\mathscr{S},\mathscr{T}) \geq \frac{r}{12}\left(C-.8\right)\alpha|V||W|-\frac{1}{2}x_2^{-1/4}\alpha|V||W|.\]
\end{lemma}

\begin{proof}

For a part $T$ of $\mathscr{T}$, let $T'$ be the set of blue vertices in $T$ in step $i$. Note that at least half of $T$ is blue (if any of its vertices are), and that in this case, they all belong to the same part $Y$ of $\mathscr{Q}_i$. Recall that for each $S \in \mathscr{S}$, we defined $i_S$ to be the largest $i$ for which $S$ has blue vertices in step $i$. We classify the parts $S \in \mathscr{S}$ into three types:
\begin{enumerate}[A.]

\item $i_S\ge i$.
\item $i_S<i$ and $S_{i_s}$ contains at most $\frac{5}{6}|S|$ blue vertices.
\item $i_S<i$ and $S_{i_s}$ contains more than $\frac{5}{6}|S|$ blue vertices.
\end{enumerate}

For each $2 \le j \le i$, let $\mathscr{S}_j^0$ be the collection of $S$ of type A, or of type B with $i_S \ge j$. Let $\mathscr{S}_j^1$ be the collection of $S$ of type C with $i_S=j-1$. Let $\mathscr{S}_j$ be the union of these two collections. Let $R_j^0$ be the set of vertices in $R_j$ that belong to some $S \in \mathscr{S}_j^0$ (recall that $R_j$ is the set of vertices that became red in step $j$). Let $R_j^1$ be the set of vertices in $V$ that belong to some $S \in \mathscr{S}_j^1$. Note that it is not necessarily true that the vertices in $R_j^1$ are in $R_j$, however, all vertices in $R_j^1$ will be red by step $j$ (so they will be in $\widetilde{R}_j$), and $R_j^1$ for different $j$ will be disjoint from each other (and from the various $R_j^0$). Also, note that while $R_j$ is a subset of $V \cup W$, $R_j^0$ and $R_j^1$ are subsets of just $V$.

\begin{figure}[h]
\centering

\begin{tikzpicture}
\draw (4,4) ellipse (1.5 and 5); 
\node at (4,9.5) {\LARGE $X$}; 
\node at (9,9.5) {\LARGE $Y$}; 
\node at (2,7) {\LARGE $X_Y^1$}; 
\node at (2,1) {\LARGE $X_Y^0$}; 
\node at (11,7) {\LARGE $Y_X^1$}; 
\node at (11,1) {\LARGE $Y_X^0$};
\draw (2.5,4) -- (5.5,4);
\draw (9,4) ellipse (1.5 and 5);
\draw (7.5,4) -- (10.5,4);
\draw (4.05,3.5) circle (1.4);
\node at (3.1,5) {$S_{j-1}$};
\node at (4.6,4.3) {$S_{j-1,Y}^1$};
\node at (4.6,3.6) {$S_{j-1,Y}^0$};
\draw (3.2,4) -- (3.2,2.4);
\node at (4,2.8) {$S_j$};

\end{tikzpicture}
\caption{The disk is $S_{j-1}$, the bottom right part is $S_j$, the horizontal line divides it into $S_{j-1,Y}^1$ and $S_{j-1,Y}^0$. Recall that $\widetilde{S}_j$ is $S_{j-1}-S_j$.}
\label{figure3a}
\end{figure}
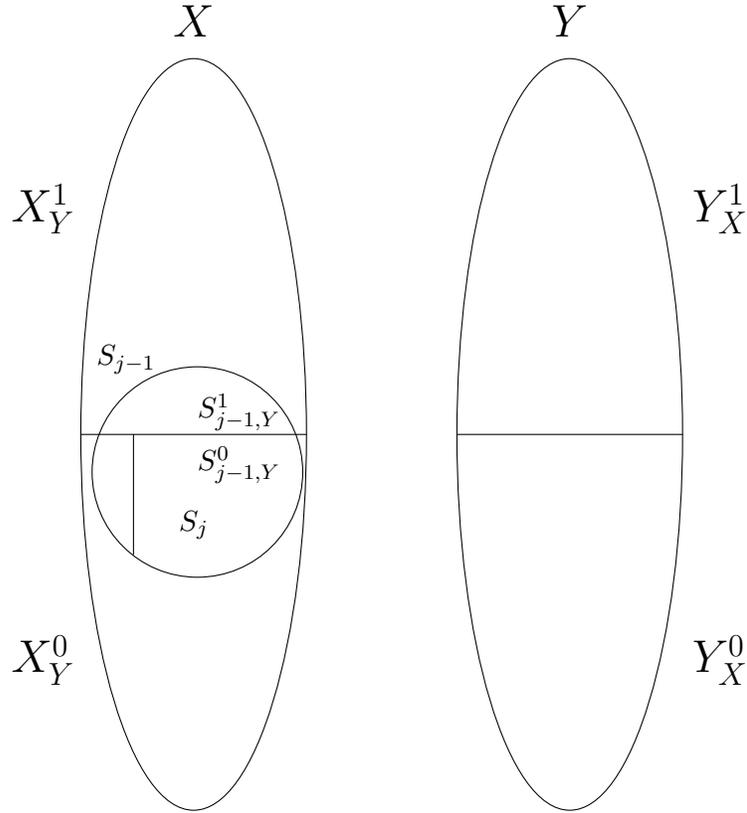

For each $2 \le j \le i$, we do the following. We know that for any $T \in \mathscr{T}$, $T'$ is a subset of a single part of $\mathscr{Q}_i$, so it is in a single part of $\mathscr{Q}_k$ for any $k \le i$. Let $Y$ be any part of $\mathscr{Q}_{j-1}$, and $T$ be such that $T' \subset Y$. 

First, suppose $S \in \mathscr{S}_j^0$, so $i_S \ge j$ and $S_j$ is nonempty. Then, $S_j$ is a subset of a part $X'\in \mathscr{P}_j$, and $S_{j-1}$ will be a subset of a part $X \in \mathscr{P}_{j-1}$ (where clearly $X' \subset X$). Assume that the pair $X$, $Y$ is active. Recall that $X$ is divided into two parts by $Y$, $X_Y^0$ and $X_Y^1$. This divides $S_{j-1}$ into two parts,  $S_{j-1,Y}^0$ and $S_{j-1,Y}^1$. The set $S_j$ is a subset of one of them (since it is inside $X'$), and, since $|S_j|>|S|/2$, it will be inside the bigger one. Recall that $S_{j-1,Y}$ is the smaller of the two. Thus, $S_j$ and $S_{j-1,Y}$ are both entirely inside one of $X_Y^0$ and $X_Y^1$, and we know that they are inside different ones. See Figure \ref{figure3a}. This means that one of them has weight $\alpha$ on all edges going to $T'$ in $G_j$, and the other has weight $-\alpha$ on all such edges. Since they both have the same weight in $\widetilde G_{j-1}$, this implies, since $S_{i_s} \subset S_j$, that
\[\left| d_{\widetilde G_j}(S_{j-1,Y},T')-d_{\widetilde G_j}(S_{i_S},T')\right|=2\alpha.\] 
In particular, we have 
\[\left|e_{\widetilde{G}_j}(S_{j-1,Y},T')-\frac{|S_{j-1,Y}|}{|S_{i_s}|}e_{\widetilde{G}_j}(S_{i_s},T')\right|=2\alpha |S_{j-1,Y}||T'|.\]
Note here that $|S_{i_s}| > \frac{|S|}{2} > |S_{j-1,Y}|$.

Observe that, if $S_j \subset X_Y^a$, then the expression within the absolute value on the left is negative if $T' \subset Y_X^a$ (and thus $S_{i_s} \subset X_Y^a$), and positive if $T' \subset Y_X^{1-a}$.

Now, suppose $S \in \mathscr{S}_j^1$. Recall that in this case, $i_S=j-1$, so $S_{j-1}$ is nonempty, but $S_j$ is empty. Thus there is a unique $X \in \mathscr{P}_{j-1}$ such that $S_{j-1} \subset X$. Recall that $\mathscr{Q}_X$ is the set of $Y \in \mathscr{Q}_{j-1}$ such that the pair $X, Y$ is active. For each $Y \in \mathscr{Q}_{X}$, we have $S_{j-1}$ splits into $S_{j-1,Y}^0$ and $S_{j-1,Y}^1$. Recall that $S_{j-1,Y}$ is the smaller of these two sets, and $S_{j-1,Y}^*$ is the larger of these two sets. Note that $T'$ is a subset of $Y_X^b$ for $b=0$ or $b=1$. Thus, as before, we have

\[\left |e_{\widetilde{G}_{j}}(S_{j-1,Y},T')-\frac{\left|S_{j-1,Y}\right|}{\left|S_{j-1,Y}^{*}\right|} e_{\widetilde{G}_{j}}(S_{j-1,Y}^{*},T')\right |=2 \alpha \left|S_{j-1,Y}\right||T'|.\]

Fix an active pair $X,Y$ with $X \in \mathscr{P}_{j-1}$ and $Y \in \mathscr{Q}_{j-1}$. Adding the equations up for all pairs $S,T$ with $S \in \mathscr{S}_j$, $S_{i_S} \subset X$, and $T' \subset Y$, we obtain that

\begin{equation}\label{eqlongexp}\sum_{\substack{S \in \mathscr{S}_j^0 \\ S_{i_s} \subset X \\T' \subset Y}}
\left |e_{\widetilde{G}_j}(S_{j-1,Y},T')-\frac{|S_{j-1,Y}|}{|S_{i_s}|}e_{\widetilde{G}_j}(S_{i_s},T') \right|
+\sum_{\substack{S \in \mathscr{S}_j^1 \\ S_{i_s} \subset X \\T' \subset Y}} \left |e_{\widetilde{G}_{j}}(S_{j-1,Y},T')-\frac{\left|S_{j-1,Y}\right|}{\left|S_{j-1,Y}^{*}\right|} e_{\widetilde{G}_{j}}(S_{j-1,Y}^{*},T')\right |
\end{equation} 
is equal to 
\begin{equation}\label{eqshorte} 
2\alpha \sum_{\substack{S \in \mathscr{S}_j \\ S_{i_s} \subset X \\T' \subset Y}} |S_{j-1,Y}||T'|.
\end{equation}

It will therefore be helpful to find a lower bound on the sum of (\ref{eqshorte}) over all active pairs $X,Y$ with $X \in \mathscr{P}_{j-1}$ and $Y \in \mathscr{Q}_{j-1}$, or equivalently, to find a lower bound on the following sum. 

\begin{equation}\label{eqsumsum1}\sum_{\substack{X \in \mathscr{P}_{j-1} \\ Y \in \mathscr{Q}_{X}}} \sum_{\substack{S \in \mathscr{S}_j\\ S_{i_S} \subset X \\T' \subset Y}} |S_{j-1,Y}||T'|.\end{equation}

Fix an $S \in \mathscr{S}_j$, and look at the terms of (\ref{eqsumsum1}) corresponding to $S$. Let $W'$ be the set of vertices in $W$ which are blue in step $i$, so $|W'|=C|W|$. If $S \in \mathscr{S}_j^0$, then $i_S \ge j$, so 
we can apply case 1 of Lemma \ref{sumalot} to get that the terms with this $S$ sum 
to at least $|\widetilde{S}_j|(C-.8)|W|$, where we recall that $\widetilde{S}_j = S_{j-1} \setminus S_j$. If $S \in \mathscr{S}_j^1$, 
then we can apply case 2 of Lemma \ref{sumalot}, to see 
that the sum of the terms containing $S$ is at least $\frac{1}{2}(|
S_{j-1}|-|S|/2)(C-.8)|W| \ge (C-.8)\frac{|S|}{6}|W|$, since $|S_{j-1}| \ge \frac{5}{6}|S|$. Summing up 
the terms with $S \in \mathscr{S}_j^0$, we obtain $(C-.8) |
R_j^0||W|$. Summing up the terms with $S \in 
\mathscr{S}_j^1$, we obtain at least $(C-.8) \frac{|R_j^1|}{6}|W|$.

Thus, we have obtained that summing (\ref{eqlongexp}) over all active pairs $X,Y$ in step $j-1$ is at least \[2\alpha \left(|R_j^0|+\frac{|R_j^1|}{6}\right)\left(C-.8\right)|W|.\]

We would like to use Lemma \ref{regular} to show that if we change $\widetilde{G}_j$ to $G$ everywhere in \eqref{eqlongexp}, then it does not decrease by much. That is, if we look at
\begin{equation}\label{eqlongexpfinal}\sum_{\substack{S \in \mathscr{S}_j^0 \\ S_{i_s} \subset X \\T' \subset Y}}
\left |e_{G}(S_{j-1,Y},T')-\frac{|S_{j-1,Y}|}{|S_{i_s}|}e_{G}(S_{i_s},T') \right|
+\sum_{\substack{S \in \mathscr{S}_j^1 \\ S_{i_s} \subset X \\T' \subset Y}} \left |e_{G}(S_{j-1,Y},T')-\frac{\left|S_{j-1,Y}\right|}{\left|S_{j-1,Y}^{*}\right|} e_{G}(S_{j-1,Y}^{*},T')\right |,
\end{equation} this is at least the value of \eqref{eqlongexp} minus $2\cdot.9\cdot x_j^{-1/4}|X||Y|$. 

Fix $a,b \in \{0,1\}$. Let $\mathscr{T}'$ be the set of $T \in \mathscr{T}$ for which $T' \subset Y_X^b$.
For $d \in \{0,1\}$, let $\mathscr{S}'_d$ be the set of $S \in \mathscr{S}_j^d$ for which $S_{j-1,Y} \subset X_Y^a$. Let $\mathscr{S}'=\mathscr{S}'_0 \cup \mathscr{S}'_1$. For $S \in \mathscr{S}'_0$, let $S'=S_{j-1,Y}$, $S''=S_{i_s}$, and for $S \in \mathscr{S}'_1$, let $S'=S_{j-1,Y}$, $S''=S_{j-1,Y}^{*}$. Note that in both cases $|S'| \le |S''|$. Then \eqref{eqlongexp} is just summing
\[ \sum_{\substack{S \in \mathscr{S}'\\T \in \mathscr{T}'}} \left| e_{\widetilde{G}_j}(S',T')-\frac{|S'|}{|S''|} e_{\widetilde{G}_j}(S'',T')\right|. \]
over $a,b \in \{0,1\}$. The key observation is that, if $a$ and $b$ are fixed, then every term above has the same sign, and hence the above sum is equal to 
\[\left| \sum_{\substack{S \in \mathscr{S}'\\T \in \mathscr{T}'}} \left( e_{\widetilde{G}_j}(S',T')-\frac{|S'|}{|S''|} e_{\widetilde{G}_j}(S'',T')\right) \right|. \]

First, we want to say that if we change $\widetilde{G}_j$ to $G$ in the expression inside the absolute values here, it changes by at most $2\cdot .9 \cdot x_j^{-1}|X_Y^a||Y_X^b|$. This is equivalent to the statement

\[\left| \sum_{\substack{S \in \mathscr{S}'\\T \in \mathscr{T}'}} \left( e_{G-\widetilde{G}_j}(S',T')-\frac{|S'|}{|S''|} e_{G-\widetilde{G}_j}(S'',T')\right) \right| \le 2\cdot .9 \cdot x_j^{-1}|X_Y^a||Y_X^b|. \]

For each $S \in \mathscr{S}'$, let $S'''$ be a random subset of $S''$ of size $|S'|$, chosen uniformly from all subsets of this size. Then $\mathbb{E}(e_{G-\widetilde{G}_j}(S''',T'))=\frac{|S'|}{|S''|}e_{G-\widetilde{G}_j}(S'',T')$ and so the above expression (without the absolute value signs) is the expected value of 

\[ \sum_{\substack{S \in \mathscr{S}'\\T \in \mathscr{T}'}} \left( e_{G-\widetilde{G}_j}(S',T')- e_{G-\widetilde{G}_j}(S''',T')\right). \]

Let $T_0'=\bigcup_{T \in \mathscr{T}'} T'$, $S_0'= \bigcup_{S \in \mathscr{S}'} S'$, $S_0'''=\bigcup_{S \in \mathscr{S}'} S'''$. The previous sum in absolute value equals \[ \left| e_{G-\widetilde{G}_j}(S_0',T_0')- e_{G-\widetilde{G}_j}(S_0''',T_0') \right| \le 2\cdot .9 \cdot x_j^{-1/4}|X_Y^a||Y_X^b|,\] where the inequality is by Lemma \ref{regular}. This holds for any choice of $S'''$, and so it holds for the expected value. Hence, \[\left| \sum_{\substack{S \in \mathscr{S}'\\T \in \mathscr{T}'}} \left( e_{G-\widetilde{G}_j}(S',T')-\frac{|S'|}{|S''|} e_{G-\widetilde{G}_j}(S'',T')\right) \right| \le 2\cdot .9 \cdot x_j^{-1/4}|X_Y^a||Y_X^b|. \]

This implies that
\begin{align*} 
\sum_{\substack{S \in \mathscr{S}'\\T \in \mathscr{T}'}} \left| e_{G}(S',T')-\frac{|S'|}{|S''|} e_{G}(S'',T')\right| \ge& \left| \sum_{\substack{S \in \mathscr{S}'\\T \in \mathscr{T}'}} \left( e_{G}(S',T')-\frac{|S'|}{|S''|} e_{G}(S'',T')\right) \right| \\ \ge& \left| \sum_{\substack{S \in \mathscr{S}'\\T \in \mathscr{T}'}} \left( e_{\widetilde{G}_j}(S',T')-\frac{|S'|}{|S''|} e_{\widetilde{G}_j}(S'',T')\right) \right|-2\cdot .9 \cdot x_j^{-1/4}|X_Y^a||Y_X^b| \\=& \sum_{\substack{S \in \mathscr{S}'\\T \in \mathscr{T}'}} \left| e_{\widetilde{G}_j}(S',T')-\frac{\left|S'\right|}{\left|S''\right|} e_{\widetilde{G}_j}(S'',T')\right|-2\cdot .9 \cdot x_j^{-1/4} | X_Y^a | | Y_X^b| \end{align*}

Thus, the sum of \eqref{eqlongexpfinal} over each active pair $X,Y$ in step $j-1$ and $a,b \in \{0,1\}$ is at least 

\begin{equation}\label{eqp24r} 2\alpha \left(|R_j^0|+\frac{|R_j^1|}{6}\right)\left(C-.8\right)|W|-2\cdot .9 \cdot  x_j^{-1/4}\alpha|V||W|.\end{equation}


This holds for each $2 \le j \le i$. Let us look at the sum of the lower bounds for each $2 \le j \le i$. Recall that $R_i$ is the set of vertices that turn red in step $i$, and $\widetilde{R}_i$ is the set of vertices that are red in step $i$. In the following discussion, the color of a vertex refers to its color in step $i$. For an $S$ of type A, each of its red vertices belong to $R_j^0$ for some $j$, since it turned red in some step. For an $S$ of type B, each of its vertices are red. We know that $i_S$ was the last step when it still had blue vertices, however, since $|S_{i_S}| \le \frac{5}{6}|S|$, we know that at least $1/6$ of the vertices of $S$ were red by step $i_S$, so they belong to $R_j$ for some $j \le i_S$. Also, by definition, for $j \leq i_S$, $S$ will belong to $\mathscr{S}_j^0$. Thus, all the vertices that turned red by step $i_S$ will belong to $R_j^0$ for some $j \leq i_S$ (recall that $R_j^0$ is the set of vertices in $R_j$ that belong to an $S \in \mathscr{S}_j^0$). To summarize, for an $S$ of type B, at least $1/6$ of its vertices will belong to $R_j^0$ for some $j$. This means that the sum of $|R_j^0|$ over $2 \le j \le i$ is at least $1/6$ of the number of red vertices that belong to an $S$ of type A or B. If $S$ is of type C, then by definition, each vertex in $S$ is red, and belongs to $R_{i_S}^1$, and so the sum of $R_{i_S}^1$ over $2 \le j \le i$ is at least the number of red vertices in an $S$ of type C. Since each $S \in \mathscr{S}$ has type A, B, or C, we obtain that the sum of the lower bounds (\ref{eqp24r}) over $2 \le j \le i$ is at least 
\[ \frac{1}{3}\alpha|\widetilde{R}_i \cap V| \left(C-.8\right) |W|-2x_2^{-1/4}\alpha|V||W|.\] 

We will use Corollary \ref{irreghelp} to get a lower bound on the irregularity in terms of the sum of  \eqref{eqlongexpfinal} over active pairs $X,Y$ in step $j-1$, for each $2 \le j\le i$. Fix a pair $S$ and $T$. We will find subsets $S_T^0,S_T^1 \subset S$ and $T_S^0,T_S^1 \subset T$ to obtain a lower bound on the irregularity. Suppose first that $S$ is of type A or B. Thus, we have the term 
\begin{equation} \label{eqshortexpr}
\left| e_G(S_{j-1,Y},T')-\frac{|S_{j-1,Y}|}{|S_{i_S}|}e_G(S_{i_S},T')\right|
\end{equation}
for some (possible empty) set of indices $j\in J_S\subset [2,i] \cap \mathbb{Z}$. Let $J_S^+$ be the set of indices for which the expression inside the absolute values is positive, and let $J_S^-$ be the set for which it is negative. Now, if we take $S_T^0$ to be the (disjoint) union of the sets $S_{j-1,Y},j \in J_S^+$, and $S_T^1=S_{i_S}$, the sum of \eqref{eqshortexpr} over $j \in J_S^+$ is just 
\[\left| e_G(S_T^0,T')-\frac{|S_T^0|}{|S_T^1|}e_G(S_T^1,T')\right|.\] 
The same holds if instead we take $S_T^0$ to be the union of $S_{j-1,Y},j \in J_S^-$, and sum \eqref{eqshortexpr} over $j \in J_S^-$. Now, the sum of the expression over one of $J_S^+$ or $J_S^-$ is at least half the sum over all of $J_S$, so we let $S_T^0$ be the union over $J_S^+$ or $J_S^-$, whichever one has a larger sum, and $S_T^1=S_{i_S}$. Notice that $|S_T^1|=|S_{i_S}| > |S|/2$, and $S_T^0$ is disjoint from $S_T^1$, so $|S_T^0| \le |S_T^1|$. If $S$ is of type C, then $S$ can only appear if $j=i_S+1$, so there is at most one term which contains $S$ and $T$. Thus, we can just take $S_T^0=S_{i_S,Y}$ and $S_T^1=S_{i_S,Y}^*$, and in this case, by definition, $|S_T^0| \le |S_T^1|$. 
Overall, we lose a factor of $2$ in the lower bound (because of the sets $S$ of type A or B), and we obtain a choice of $S_T^0$ and $S_T^1$ with $|S_T^0| \le |S_T^1|$ for certain pairs $S$ and $T$ such that 

\[\sum_{S,T} \left |e_G(S_T^0,T')-\frac{|S_T^0|}{|S_T^1|}e_G(S_T^1,T'))\right | \ge \frac{r}{6}\left(C-.8\right)\alpha|V||W|-x_2^{-1/4}\alpha|V||W|.\]

Applying Corollary \ref{irreghelp} with $T_S^0=T_S^1=T'$, we obtain the desired lower bound on the irregularity of \[\frac{r}{12}\left(C-.8\right)\alpha|V||W|-\frac{1}{2}x_2^{-1/4}\alpha|V||W|.\] 
Note that we can apply this corollary as $|S_T^0||T_S^0| \leq |S_T^1||T_S^1|$, which follows from $T_S^0=T_S^1=T'$ and $|S_T^0| \le |S_T^1|$. This completes the proof.

\end{proof}

The next lemma establishes another lower bound on the irregularity between the two partitions. This will be useful if there is a single step in which the number of red vertices increases by a substantial amount. As this lemma obtains a bound on the irregularity by considering a single step, the proof of this lemma is simpler than that of the previous lemma.

\begin{lemma} \label{onestep}

Suppose in step $i$, we have $C|W|$ blue vertices in $W$ and $\beta|V|$ blue vertices in $V$, and in step $i+1$, we only 
have $\beta'|V|$ blue vertices in $V$, where $\beta>\frac{\beta'+1}{2}$. Then

\[\irreg(\mathscr{S},\mathscr{T}) \geq \frac{1}{4}\left(\beta-\frac{\beta'+1}{2}\right)\left(C-.8\right)\alpha|V||W| - x_{i+1}^{-1/4} \alpha |V||W|. \]
\end{lemma}
\begin{proof}
 Let $\mathscr{S}^0$ be the collection of $S \in \mathscr{S}$ with $i_S > i$, and $\mathscr{S}^1$ be the collection of those $S \in \mathscr{S}$ for which $i_S=i$. Let $\mathscr{S}^*$ be the union of these two collections, so $|\bigcup_{S \in \mathscr{S}^*} S|=\beta|V|$. For each $S \in \mathscr{S}^*$ we have $i_S \ge i$, and let $X(S)$ be the part $X \in \mathscr{P}_i$ with $S_i \subset X$. Recall that for any $X \in \mathscr{P}_i$, $\mathscr{Q}_X$ is the set of $Y \in \mathscr{Q}_i$ such that $X,Y$ is an active pair, and, if $X=X(S)$, for any $Y \in \mathscr{Q}_X$, $S_i$ is split into 
$S_{i,Y}^0$ and $S_{i,Y}^1$. Further recall that we defined $S_{i,Y}$ to be the smaller of the two parts and $S_{i,Y}^*$ to be the larger of the two parts. Similarly, for 
each $T$, if $T'$ is the set of vertices in $T$ which are blue in step $i$, we have $T' \subset Y \in \mathscr{Q}_i$ for some $Y$, and for each $X \in \mathscr{P}_i$, $T'$ is split 
into two parts. Let $T'_X$ be the bigger of the two parts. Now, as in the proof of the previous lemma, we have

\begin{equation}\label{eqlong:lm}
\sum_{\substack{X \in \mathscr{P}_i\\Y \in \mathscr{Q}_X}} \sum_{\substack{S \in \mathscr{S}^*:S_i \subset X\\T \in \mathscr{T}:T' \subset Y}}\left|e_{\widetilde{G}_{i+1}}(S_{i,Y},T'_X)-\frac{|S_{i,Y}|}{|S_{i,Y}^{*}|} e_{\widetilde{G}_{i+1}}(S_{i,Y}^{*},T'_X)\right|\\
=2 \alpha \sum_{\substack{X \in \mathscr{P}_i\\Y \in \mathscr{Q}_X}} \sum_{\substack{S \in \mathscr{S}^*:S_i \subset X\\T \in \mathscr{T}:T' \subset Y}}|S_{i,Y}||T_X'|.
\end{equation}

By definition, for each $T$ and $X$, we have $|T'_X| \ge |T'|/2$. Thus, 

\begin{equation}\label{eq562} 2 \alpha \sum_{\substack{X \in \mathscr{P}_i\\Y \in \mathscr{Q}_X}} \sum_{\substack{S \in \mathscr{S}^*:S_i \subset X\\T \in \mathscr{T}:T' \subset Y}}|S_{i,Y}||T_X'| \ge \alpha \sum_{\substack{X \in \mathscr{P}_i\\Y \in \mathscr{Q}_X}} \sum_{\substack{S \in \mathscr{S}^*:S_i \subset X\\T \in \mathscr{T}:T' \subset Y}}|S_{i,Y}||T'|.\end{equation}

Applying Lemma \ref{sumalot} with $j=i+1$ and $W'$ the subset of $W$ of blue vertices in step $i$, we obtain

\begin{align}
\nonumber \sum_{\substack{X \in\mathscr{P}_i\\Y \in \mathscr{Q}_X}} \sum_{\substack{S \in \mathscr{S}^*:S_i \subset X\\T \in \mathscr{T}:T' \subset Y}}|S_{i,Y}||T'| =& \sum_{S \in \mathscr{S}^*}\sum_{Y \in \mathscr{Q}_{X(S)}} |S_{i,Y}||Y'| =\sum_{S \in \mathscr{S}^0} \sum_{Y \in \mathscr{Q}_{X(S)}} |S_{i,Y}||Y'|+\sum_{S \in \mathscr{S}^1} \sum_{Y \in \mathscr{Q}_{X(S)}} |S_{i,Y}||Y'| \\
\nonumber \ge& \sum_{S \in \mathscr{S}^0} \left(C-.8\right)|\widetilde{S}_{i+1}||W|+\sum_{S \in \mathscr{S}^1} \frac{1}{2}\left(|S_i|-\frac{|S|}{2}\right)\left(C-.8\right)|W|\\
=&\left(\sum_{S \in \mathscr{S}^0}|\widetilde{S}_{i+1}|+\sum_{S \in \mathscr{S}^1}\frac{|S_i|}{2}- \sum_{S \in \mathscr{S}^1}\frac{|S|}{4}\right)\left(C-.8\right)|W|. \label{eqlonglong}
\end{align}

Note that $\sum_{S \in \mathscr{S}^0} |\widetilde{S}_{i+1}|+\sum_{S \in \mathscr{S}^1} |S_i|$ is exactly the number of \emph{new} red vertices in step $i+1$. Indeed, any $S$ that has blue vertices still in step $i$ is in one of $\mathscr{S}^0$ or $\mathscr{S}^1$. If $S$ is in $\mathscr{S}^0$, then it contains blue vertices in step $i+1$, and so the new red vertices from $S$ are $S_i \setminus S_{i+1}=\widetilde{S}_{i+1}$. If $S$ is in $\mathscr{S}^1$, then all of its vertices are red in step $i+1$. Now, each vertex that belongs to an $S \in \mathscr{S}^1$ is red in step $i+1$, thus the number of such vertices is at most $(1-\beta')|V|$. That is, $\sum_{S \in \mathscr{S}^1} |S| \le (1-\beta')|V|$. This gives 
\begin{equation}\sum_{S \in \mathscr{S}^0}|\widetilde{S}_{i+1}|+\sum_{S \in \mathscr{S}^1}\frac{|S_i|}{2}- \sum_{S \in \mathscr{S}^1}\frac{|S|}{4} \ge \frac{1}{2}\left( \sum_{S \in \mathscr{S}^0}|\widetilde{S}_{i+1}|+\sum_{S \in \mathscr{S}^1}|S_i|- \sum_{S \in \mathscr{S}^1}\frac{|S|}{2} \right) \ge \frac{1}{2}\left(\beta-\beta'-\frac{1-\beta'}{2} \right).\label{longlineeq} \end{equation}

Combining (\ref{eqlonglong}) and (\ref{longlineeq}), and simplifying, we obtain 
 
\begin{equation*}\sum_{\substack{X \in \mathscr{P}_i\\Y \in \mathscr{Q}_X}} \sum_{\substack{S \in \mathscr{S}^*:S_i \subset X\\T \in \mathscr{T}:T' \subset Y}}|S_{i,Y}||T'| \ge \frac{1}{2}\left(\beta-\beta'-\frac{1-\beta'}{2} \right)(C-.8)|V||W| =\frac{1}{2}\left(\beta-\frac{1+\beta'}{2}\right)\left(C-.8\right)|V||W|.\end{equation*}

Also using (\ref{eqlong:lm}) and (\ref{eq562}), we obtain 
\begin{equation*}\sum_{\substack{X \in \mathscr{P}_i\\Y \in \mathscr{Q}_X}} \sum_{\substack{S \in \mathscr{S}^*:S_i \subset X\\T \in \mathscr{T}:T' \subset Y}}\left|e_{\widetilde{G}_{i+1}}(S_{i,Y},T'_X)-\frac{\left|S_{i,Y}\right|}{\left|S_{i,Y}^{*}\right|} e_{\widetilde{G}_{i+1}}(S_{i,Y}^{*},T'_X)\right| \ge \frac{1}{2}\left(\beta-\frac{1+\beta'}{2}\right)\left(C-.8\right)\alpha|V||W|.\end{equation*}

As in the proof of the previous lemma, we want to change $\widetilde{G}_{i+1}$ to $G$ in the above estimate, which has a small  effect on the lower bound. As the argument is the same as in the proof of Lemma \ref{addingup}, we only summarize it. Again, for each term, the subsets of $S$ and of $T$ that we are considering are completely within $X_Y^a$ and $Y_X^b$, respectively, for an active pair $X,Y$, and some $a,b \in \{0,1\}$. Fix $X,Y,a,b$, and look at the terms corresponding to $X_Y^a$ and $Y_X^b$. As before, all these terms will have the same sign, and the exact same argument shows, using Lemma \ref{regular}, that changing these terms decreases the sum by at most $2\cdot .9 \cdot  \alpha x_{i+1}^{-1/4}|X_Y^a||Y_X^b|$. For simplicity we replace the $.9$ with $1$, and since this is true for all $X,Y,a,b$, adding it up, we have

\begin{multline*}
\sum_{\substack{X \in \mathscr{P}_i\\Y \in \mathscr{Q}_X}} \sum_{\substack{S \in \mathscr{S}^*:S_i \subset X\\T \in \mathscr{T}:T' \subset Y}}\left|e_{G}(S_{i,Y},T'_X)-\frac{\left|S_{i,Y}\right|}{\left|S_{i,Y}^{*}\right|} e_{G}(S_{i,Y}^{*},T'_X)\right|\\
\ge \frac{1}{2}\left(\beta-\frac{1+\beta'}{2}\right)\left(C-.8\right)\alpha|V||W|-2x_{i+1}^{-1/4}\alpha|V||W|.
\end{multline*}

Let $\mathscr{R}\subset \mathscr{S} \times \mathscr{T}$ be the set of pairs $S,T$ for which the corresponding pair $X(S),Y(T)$ is active. These are precisely the pairs $S,T$ that appear in the previous sum. For any such $(S,T) \in \mathscr{R}$, let $X,Y$ be the corresponding active pair, and take $T_S^0=T_S^1=T_X'$, $S_T^0=S_{i,Y}$, $S_T^1=S_{i,Y}^*$. By definition, $|S_T^0| \le |S_T^1|$, and we can apply Corollary \ref{irreghelp} to obtain
\[\irreg(\mathscr{S},\mathscr{T}) \geq \frac{1}{4}\left(\beta-\frac{1+\beta'}{2}\right)\left(C-.8\right)\alpha|V||W|-x_{i+1}^{-1/4}\alpha|V||W|.\]
This completes the proof.
\end{proof}

\begin{proof}[Proof of Lemma \ref{withassumption}]
Now, we show that, assuming $\mathscr{S}$ refines $\mathscr{P}_1$, and $\mathscr{T}$ refines $\mathscr{Q}_1$, they must be close to 
being refinements of $\mathscr{P}_s$ and $\mathscr{Q}_s$, and cannot have too few parts. Let $v_i$ and $w_i$ be the proportion of vertices that are blue in step $i$ in $V$ and $W$, respectively. 
Then, since vertices can turn red but can never turn blue, this gives us two non-increasing sequences. We want to show that if the 
irregularity is less than $\frac{1}{5000}\alpha|V||W|$, then both sequences stay above $.97$, so in particular $v_s$ and $w_s$ are at least $.97$.
If there is a step where $w_i \ge .9$, but $v_i \le .97$, then using Lemma \ref{addingup}, and recalling that $x_2=2^{2^6}$, the irregularity is at least \[\frac{.03}{12}(.9-.8)\alpha|V||W|-\frac{1}{2}x_2^{-1/4}\alpha|V||W|=\frac{1}{4000}\alpha|V||W|-\frac{1}{2}x_2^{-1/4}\alpha|V||W| \ge \frac{1}{5000}\alpha|V||W|.\] By the same logic, this also holds if $v_i \ge .9$ and $w_i \le .97$. Thus, the only way it is possible that one of $v_s,w_s$ is below $.97$ is if there is an $i$ such that $v_i$ and $w_i$ are both at least $.97$, but $v_{i+1}$ and $w_{i+1}$ are both at most $.9$. Then, using just the fact that $w_i \ge .9$, $v_i \ge .97$, $v_{i+1} \le .9$, we can apply Lemma \ref{onestep} to obtain that the irregularity is at least \[\frac{1}{4}\left(.97-\frac{.9+1}{2}\right)(.9-.8)\alpha|V||W|-.9x_{i+1}^{-1/4}\alpha|V||W|=\frac{1}{2000}\alpha|V||W|-.9x_{i+1}^{-1/4}\alpha|V||W| \ge \frac{1}{5000} \alpha |V||W|.\] 

Thus, if the irregularity is less than $\frac{1}{5000}\alpha|V||W|$, then both $v_s$ and $w_s$ are at least $.97$, and so $\mathscr{S}$ and $\mathscr{T}$ are in fact close to being refinements of $\mathscr{P}_s$ and $\mathscr{Q}_s$ respectively. To each part $S \in \mathscr{S}$ that has blue vertices, there is a unique part  $X \in \mathscr{P}_s$ with $S_s \subset X$. In other words, if two blue vertices belong to different parts of $\mathscr{P}_s$, then they belong to different parts of $\mathscr{S}$. Since at least a $.97$ proportion of the vertices are blue, at least $.97|\mathscr{P}_s|$ parts of $\mathscr{P}_s$ have blue vertices (since each part in $\mathscr{P}_s$ has the same size), and so we have at least $.97|\mathscr{P}_s|$ parts in $\mathscr{S}$. Applying the analogous argument, we have at least $.97|\mathscr{Q}_s|$ parts in $\mathscr{T}$. 

\end{proof}

Now, in order to finish the proof of the main result, we need the following lemma.

\begin{lemma} \label{divide}
Let $G$ be a graph with a vertex partition $\mathscr{P}$, and $\mathscr{Q}$ is a refinement of $\mathscr{P}$ such that each part of $\mathscr{P}$ is divided into at most $k$ parts. If $\mathscr{Q}$ has irregularity $I$, then $\mathscr{P}$ has irregularity at least $\frac{I}{2k^2}$.
\end{lemma}

\begin{proof}
Let $X$ and $Y$ be two parts of $\mathscr{P}$. The idea is that since we divide each part into at most $k$ parts, we can find a pair of parts of $\mathscr{Q}$ in $X$ and $Y$ that have large irregularity, and use this to show that $X$ and $Y$ have some irregularity. To be precise, let $\mathscr{X}$ be the parts of $\mathscr{P}'$ that are subsets of $X$, and $\mathscr{Y}$ be the parts of $\mathscr{P}'$ that are subsets of $Y$. Define
\[I_{XY}=\sum_{\substack{X' \in \mathscr{X}\\ Y' \in \mathscr{Y}}} \irreg(X',Y').\]

Now, since each part of $\mathscr{P}$ is divided into at most $k$ parts, one of the terms of this sum is at least $I_{XY}/k^2$. Let $X'$ and $Y'$ be a pair of elements of $\mathscr{X}$ and $\mathscr{Y}$ for which this holds. Now, this means that there are subsets $X''\subset X',Y'' \subset Y'$ such that

\[ \left |e(X'',Y'')-\frac{|X''||Y''|}{|X'||Y'|}e(X',Y')\right |\ge \frac{I_{XY}}{k^2}. \]

Now, we can apply Lemma \ref{irreghelpsinglepair} to show that \[\irreg(X,Y) \ge \frac{I_{XY}}{2k^2}.\] 

Adding this up for all $X$ and $Y$, we obtain

\begin{align*} \irreg(\mathscr{P}) = &  \sum_{X,Y \in \mathscr{P}} \irreg(X,Y) \ge \sum_{X,Y \in \mathscr{P}} \frac{I_{XY}}{2k^2} = \frac{1}{2k^2} \sum_{X,Y \in \mathscr{P}} \sum_{\substack{X' \in \mathscr{X}\\Y' \in \mathscr{Y}}} \irreg(X',Y') \\ 
= & \frac{1}{2k^2} \sum_{X',Y' \in \mathscr{P}'} \irreg(X',Y')=\frac{1}{2k^2} \irreg(\mathscr{P}') .\end{align*}

\end{proof}

\begin{proof}[Proof of Theorems 3.5 and 1.2]

Consider a partition $\mathscr{P}$ with less than $\frac{.97}{4x_1}(|\mathscr{P}_s|+|\mathscr{Q}_s|)$ parts. Let $\mathscr{P}'$ be the common refinement of $\mathscr{P}$ and $\mathscr{P}_1 \cup \mathscr{Q}_1$. Then $\mathscr{P}'$ divides each part of $\mathscr{P}$ into at most $4x_1$ parts. It follows that $\mathscr{P}'$ has less than $4x_1(\frac{.97}{4x_1})(|\mathscr{P}_s|+|\mathscr{Q}_s|)=.97(|\mathscr{P}_s+|\mathscr{Q}_s|)$ parts. By Lemma \ref{withassumption}, $\mathscr{P}'$ has irregularity at least $\frac{1}{5000}\alpha|V||W|$. By Lemma $\ref{divide}$, $\mathscr{P}$ has irregularity at least $\frac{1}{5000}\frac{1}{32x_1^2}\alpha|V||W|= \frac{1}{2^{24} \cdot 10000}\alpha|V||W| > \epsilon (|V|+|W|)^2$ (recall that $x_1=2^{10}$ and $\alpha > 2^{26}10^4 \epsilon$). This is a contradiction, and so we have completed the proof of Theorem 
\ref{maingen}. 

Now, to finish the proof of Theorem 1.2, we first prove by induction that $x_k \ge 2^8T(k+1)$ holds for each positive integer $k$. Indeed, this is true for the base case $k=1$, since $x_1 = 2^{10} = 2^8T(2)$. Since $x_k \ge x_1 = 2^{10}$, we have $x_k/16 \ge 2^{-8}x_k+8.$ So, by the induction hypothesis, we have \[x_{k+1}=2^{x_k/16} \ge 2^{2^{-8} x_k+8} \ge 2^{T(k+1)+8}=2^8 T(k+2), \] proving the claimed bound. 

By Theorem \ref{maingen}, the number of parts any vertex partition with irregularity at most $\epsilon |V(G)|^2$  is at least $\frac{.97}{4x_1}(|\mathscr{P}_s|+|\mathscr{Q}_s|) > \frac{1}{4x_1}|\mathscr{P}_s|=2^{-12}|\mathscr{P}_s| \geq 2^{-12}x_s \geq 2^{-12}2^{8}T(s+1) \geq T(s)$, where the last inequality holds for $s \geq 3$. As $s \ge 10^{-26}\epsilon^2$, this proves Theorem \ref{main}.
\end{proof}

\section{Concluding Remarks}

In this paper, we determined that the number $M(\epsilon)$ of parts in Szemer\'edi's regularity lemma, Theorem \ref{regver}, grows as a tower of twos of height $\Theta(\epsilon^{-2})$. The upper bound we proved is a tower of height $2+\epsilon^{-2}/16$, and the lower bound is a tower of twos of height $10^{-26}\epsilon^{-2}$. We next sketch how to improve the constant factors in both the upper and lower bounds. 

We first discuss the upper bound, whose proof is given in Section \ref{upper}. Recall that the range of possible values for the mean square density is $[d^2,d]$, where $d$ is the edge density of the graph. To maximize the length of this interval, we should have $d=1/2$. One can easily check that the upper bound argument is tight  only if for every pair of parts $X,Y$ in a partition at a certain step, there are subsets $X' \subset X$ and $Y' \subset Y$ such that $|X'||Y'|=|X||Y|/2$ and $\left|e(X',Y')-d(X,Y)|X'|Y'|\right|=\epsilon |X||Y|$. It would then follow that $|d(X',Y')-d(X,Y)|=2\epsilon$. The distribution of densities between pairs of parts in the partition in step $i$ would then follow a binomial distribution with mean $1/2$ and $i$ steps of size $2\epsilon$. This is not feasible, as the densities are in the interval $[0,1]$, and after $\Omega(\epsilon^{-2})$ steps, a
constant fraction of the pairs go between pairs of parts whose density would be outside the interval $[0,1]$. 
Optimizing this argument to get a best possible upper bound would be more complicated. We also remark that this issue of densities having to be between $0$ and $1$ is the same as one of the main technical difficulties of the lower bound proof, the fact that there are inactive pairs. 

We next discuss the lower bound. Various parts of our lower bound proof do not attempt to optimize the constant factor in the tower height bound for $M(\epsilon)$. One particularly large loss (a factor $x_1^4$, where $x_1$ is the size of the first partition in the construction) comes from picking a large $x_1$ and assuming that the partition in question is a refinement of the partition in the first step. Instead, we could pick a first partition with $x_1$ parts, and pick the graph in the first step to have densities bounded away from $0$ and $1$ and have the property that any refinement with small irregularity is close to being a refinement of this given partition. We chose the current presentation without optimal constant factors for the sake of clarity. 

\noindent {\bf The original version of Szemer\'edi's regularity lemma}

The original version of the regularity lemma can be stated in a more general form as follows. We say a pair $X,Y$ of vertex subsets form an $(\epsilon,\delta)$-regular pair if there is $\alpha \in [0,1]$ such that for any subsets $U\subset X$, $W \subset Y$, such that $|U| \ge \delta |X|$, $|W| \ge \delta |Y|$, we have $d(U,W)$ lies between $\alpha$ and $\alpha+\epsilon$. Now, given $\epsilon, \delta, \eta>0$, call a partition into $k$ parts $(\epsilon, \delta, \eta)$-regular if all but at most $\eta k^2$ pairs of parts are $(\epsilon,\delta)$-regular. The lemma states that for any $\epsilon,\delta,\eta$, there is a $K=K(\epsilon,\delta,\eta)$ such that there is an equitable partition into $k \le K$ parts that is $(\epsilon,\delta,\eta)$-regular. Again, it is not too difficult to see that if the irregularity of an equitable partition is at most $\frac{1}{2}\epsilon \delta^2 \eta |V|^2$, then it is $(\epsilon,\delta,\eta)$-regular. Conversely,  if $\epsilon=\delta=\eta$, then if an equitable partition is $\epsilon$-regular, it has irregularity at most $2\epsilon|V|^2$.

Gowers \cite{Go97} showed with $c=1/16$ that there is a graph $G$ whose smallest equitable $(1-\delta^{c},\delta,1-\delta^c)$-regular partition is at least a tower of height $\delta^{-c}$. Addressing a question of Gowers, Conlon and Fox \cite{CF12} showed that there are absolute constants $\epsilon,\delta>0$ such that there is a graph $G$ whose smallest equitable $(\epsilon,\delta,\eta)$-regular partition is at least a tower of height $\Omega(\eta^{-1})$, which determines the right tower height as a function of $\eta$. 

In the case $\epsilon=\delta=\eta$, the upper bound proof gives a tower height of $O(\epsilon^{-5})$. 
Our result in Theorem \ref{main} directly implies a new lower bound on the tower height in this case of $\Omega(\epsilon^{-2})$. In fact, one can combine the construction ideas in Theorem \ref{main} and that in \cite{CF12}, and one should get a tower height of $\Omega(\epsilon^{-3})$. The addtional idea from \cite{CF12} is that instead of adding and subtracting densities for every active pair at each step, we only do this for a random fraction (on the order of $\eta$) of the active pairs at each step. A further variant using ideas similar to those of Gowers by incorporating $\delta$ into the construction should improve the exponent beyond $3$, although we don't see yet how to match the upper bound exponent of $5$.

\vspace{0.1cm}

\noindent {\bf Acknowledgement:} We would like to thank Noga Alon and David Conlon for helpful discussions.

\bibliography{ref}

\bibliographystyle{amsplain}








\end{document}